\documentclass[12pt]{article}

\usepackage{mathpazo}
\usepackage{eulervm}

\usepackage{amsmath, amssymb, amsthm}
\numberwithin{equation}{section}
\usepackage{geometry}
\usepackage{enumerate}

\usepackage{hyperref}

\newtheorem{theorem}{Theorem}[section]
\newtheorem{lemma}[theorem]{Lemma}
\newtheorem{fact}[theorem]{Fact}
\newtheorem{proposition}[theorem]{Proposition}
\newtheorem{corollary}[theorem]{Corollary}
\theoremstyle{definition}
\newtheorem{definition}[theorem]{Definition}
\theoremstyle{remark}
\newtheorem{remark}[theorem]{Remark}

\theoremstyle{plain}
\newtheorem{thmx}{Theorem}

\title{%
  {\bfseries A Tale of Two Idempotents}\\[0.35em]
  {\large\normalfont Casselman--Shalika and Spherical Genericity}%
}
\author{Yi Luo}
\date{}
\hypersetup{pdftitle={A Tale of Two Idempotents: Casselman-Shalika and Spherical Genericity},
  pdfauthor={Yi Luo}}

\begin{document}
\maketitle

\begin{abstract}
We give a self-contained Hecke-algebraic derivation of the Casselman--Shalika formula and J.-S.\ Li's genericity criterion for irreducible spherical representations of unramified groups, without using the unramified principal series or intertwining operators. The argument centers on the spherical and sign idempotents $e_K$ and $e_{\mathrm{sgn}}$ of the Iwahori--Hecke algebra $\mathcal H$. Their left ideals $\mathcal He_K=\mathcal Ae_K$ and $\mathcal He_{\mathrm{sgn}}=\mathcal Ae_{\mathrm{sgn}}$ are free of rank one over the Bernstein subalgebra $\mathcal A$. Describing $e_K\mathcal He_K$ inside $\mathcal Ae_K$ recovers the Satake isomorphism. Describing $e_K\mathcal He_{\mathrm{sgn}}$ inside $\mathcal Ae_{\mathrm{sgn}}$ yields rank-one freeness of the $K$-invariants of the Gelfand--Graev representation and the Casselman--Shalika formula. Symmetrically, describing $e_{\mathrm{sgn}}\mathcal He_K$ inside $\mathcal Ae_K$ determines when the sign-isotypic part of a spherical module is non-zero, and hence yields Li's genericity criterion.
\end{abstract}


\section{Introduction}

Let $F$ be a $p$-adic field with ring of integers $\mathcal{O}$, maximal ideal $\mathfrak{p}$, and residue cardinality $q$, and let $G$ be an unramified connected reductive group over $F$. Fix a Borel subgroup $B = TU$ defined over $F$, a maximal $F$-split torus $S \subseteq T$, the opposite Borel $\overline{B} = T\overline{U}$, a hyperspecial maximal compact subgroup $K$, and an Iwahori subgroup $I \subseteq K$.

The two idempotents of the title are the spherical idempotent $e_{K}$ and the sign idempotent $e_{\mathrm{sgn}}$ of the Iwahori--Hecke algebra $\mathcal{H}$. The idempotent $e_{K}$ governs spherical representations, whereas $e_{\mathrm{sgn}}$ governs Whittaker models. We compute $e_{K}\mathcal{H}e_{K}$ for every unramified group. Under the hypothesis $\rho^\vee\in X$, we also compute the mixed products $e_{K}\mathcal{H}e_{\mathrm{sgn}}$ and $e_{\mathrm{sgn}}\mathcal{H}e_{K}$, whose representation-theoretic consequences extend to every unramified group. The derivation is self-contained and uses only the Iwahori--Matsumoto relations, the Bernstein presentation, and the results on Iwahori invariants collected in Section \ref{sec:iwahori_fixed}. It requires neither intertwining operators nor the unramified principal series. These computations yield the Satake isomorphism, the Casselman--Shalika formula together with the underlying freeness theorem, and the genericity criterion for spherical representations.

\subsection*{Main results}

To state the results, write $X := X_*(S)$ for the cocharacter lattice of $S$, $X^+$ for its dominant cone, and $t_y := y(\varpi^{-1}) \in T(F)$ for $y \in X$, where $\varpi$ is a uniformizer of $F$. Let $W$ be the relative Weyl group, $\Sigma$ the reduced root system of non-multipliable relative roots, and $\rho^\vee \in X \otimes_{\mathbb{Z}} \mathbb{Q}$ the half-sum of the positive absolute coroots. Let $\mathcal{H} := \mathcal{H}(G, I)$ be the Iwahori--Hecke algebra, with Haar measure normalized by $\mathrm{vol}(I) = 1$. Its Bernstein presentation, available for every connected reductive group \cite{Rostami}, exhibits the commutative Bernstein subalgebra $\mathcal{A} = \mathrm{span}_{\mathbb{C}}\{\theta_\lambda : \lambda \in X\} \cong \mathbb{C}[X]$, on which $W$ acts. Write $T_w := \mathbf{1}_{IwI}$ for the standard basis elements of $\mathcal{H}$ and $q_w := [IwI : I]$, and set $W(q) := \sum_{w \in W} q_w$ and $W(q^{-1}) := \sum_{w \in W} q_w^{-1}$. Explicitly, the two idempotents are
$$e_{K} \,=\, \frac{1}{W(q)}\sum_{w \in W} T_{w}, \qquad e_{\mathrm{sgn}} \,=\, \frac{1}{W(q^{-1})}\sum_{w \in W} (-1)^{\ell(w)}\,q_{w}^{-1}\, T_{w}.$$
For non-split unramified groups, unequal Hecke parameters may occur. Each simple root $a \in \Sigma$ carries the parameter $q_a := [I s_a I : I]$ and a second Bruhat--Tits parameter $q_{a*}$, attached to the affine wall $\{a = 1\}$. One has $q_{a*} = q_a$ unless $a$ is divisible in the relative root system $\Phi$. Set $c_1(a) := (q_a q_{a*})^{1/2}$ and $c_2(a) := (q_a/q_{a*})^{1/2}$. All four parameters are extended $W$-invariantly to $\Sigma$ (Section \ref{sec:notation}).

For $y \in X^+$, the Weyl character
\begin{equation*}
\mathrm{ch}V_y := \frac{\mathrm{Alt}\big(\theta_{y + \rho^\vee}\big)}{\mathrm{Alt}\big(\theta_{\rho^\vee}\big)}, \qquad \mathrm{Alt}(F) := \sum_{w \in W}(-1)^{\ell(w)}\, w(F),
\end{equation*}
is a well-defined element of $\mathcal{A}^W$ even when $\rho^\vee \notin X$ (Fact \ref{fact:chV}(2)). As $y$ runs over $X^+$, these elements form a $\mathbb{C}$-basis of $\mathcal{A}^W$. For split $G$, $\mathrm{ch}V_y$ is the irreducible character of the dual group with highest weight $y$. For general unramified $G$, see Remark \ref{rem:chV_interpretation}.

Fix a character $\omega$ of $\overline{U}$ of \emph{conductor $\mathcal{O}$}: trivial on $\overline{U} \cap K$, with non-trivial restriction to the next larger step of the root-group filtration along each simple root (Section \ref{sec:notation}). Its values are roots of unity, so the complex conjugate $\overline{\omega}$ equals $\omega^{-1}$. The Gelfand--Graev representation attached to the conjugate character $\overline{\omega}$ is the compact induction $\mathrm{ind}_{\overline{U}}^{G}\,\overline{\omega}$, and the spherical Whittaker functions below belong to the smooth induction $\mathrm{Ind}_{\overline{U}}^{G}\,\omega$. For $y \in X^+$, let $\Phi_y \in (\mathrm{ind}_{\overline{U}}^{G}\,\overline{\omega})^K$ be the function supported on $\overline{U}t_yK$ with $\Phi_y(u\,t_y\,k) = \delta_{\overline{B}}(t_y)^{1/2}\,\overline{\omega}(u)$. These functions form a basis of the $K$-invariants by Lemma \ref{lem:whittaker_support}(2).

The Casselman--Shalika formula and J.-S.\ Li's genericity criterion are classical. Our contribution is a uniform Hecke-algebraic derivation, including the unequal-parameter identity below and its applications to arbitrary unramified groups. The mixed-product identities in Theorems \ref{thmx:CS} and \ref{thmx:genericity} assume $\rho^\vee \in X$, while their representation-theoretic statements hold for every unramified $G$.

Theorem \ref{thm:spherical_iso} proves the classical Satake isomorphism directly from the Bernstein cross-relation.

\begin{thmx}\label{thmx:satake}
$e_{K}\mathcal{H}e_{K} \,=\, \mathcal{A}^{W}e_{K}$, and $F \mapsto Fe_{K}$ is an algebra isomorphism from $\mathcal{A}^{W}$ onto the spherical Hecke algebra $\mathcal{H}_{K} := e_{K}\mathcal{H}e_{K}$.
\end{thmx}

The second theorem is the Casselman--Shalika formula with its underlying freeness, proved in Proposition \ref{prop:eKHe_sgn} and Theorems \ref{thm:bbf_generalized}, \ref{thm:freeness_general}, and \ref{thm:CS_general}.

For $a\in\Sigma$, set
$$f_a := \begin{cases} 1 - q_a\,\theta_{-a^\vee} & \text{if } a \text{ is not divisible},\\[2pt] \big(1 - c_1(a)\,\theta_{-a^\vee}\big)\big(1 + c_2(a)\,\theta_{-a^\vee}\big) & \text{if } a \text{ is divisible}\end{cases}.$$

\begin{thmx}\label{thmx:CS}
If $\rho^\vee\in X$, then
$e_{K}\mathcal{H}e_{\mathrm{sgn}} \,=\, \mathcal{A}^{W}\,\Omega^{+}e_{\mathrm{sgn}}$, where
$$\Omega^{+} = \theta_{\rho^\vee} \prod_{a \in \Sigma^+} f_a,$$
and explicitly
\begin{equation}\label{eq:bbf_intro}
e_{K}\,F\,e_{\mathrm{sgn}} = \frac{1}{W(q)}\cdot\frac{\mathrm{Alt}(F)}{\mathrm{Alt}\big(\theta_{\rho^{\vee}}\big)}\,\Omega^{+}\, e_{\mathrm{sgn}}, \qquad F \in \mathcal{A}.
\end{equation}
For every unramified $G$, writing $h_y := (\mathrm{ch}V_y)e_K \in \mathcal{H}_K$, one has
$$\Phi_0 * h_y = \Phi_y, \qquad y \in X^+,$$
and $(\mathrm{ind}_{\overline{U}}^{G}\,\overline{\omega})^K$ is a free right $\mathcal{H}_K$-module of rank one generated by $\Phi_0$. For every algebra character $\chi$ of $\mathcal{H}_K$, every $0 \neq \mathcal{W} \in (\mathrm{Ind}_{\overline{U}}^{G}\,\omega)^K$ with $h \cdot \mathcal{W} = \chi(h)\,\mathcal{W}$ for all $h \in \mathcal{H}_K$ satisfies $\mathcal{W}(t_y) = 0$ for $y \notin X^+$, and
\[
\mathcal{W}(t_y) = \delta_{\overline{B}}(t_y)^{1/2}\,\chi(h_y)\,\mathcal{W}(e) \qquad (y \in X^+).
\]
In particular, $\mathcal{W}(e) \neq 0$.
\end{thmx}

The third theorem is the genericity criterion, proved in Proposition \ref{prop:esgnHeK}, Theorem \ref{thm:genericity_Z}, and Theorem \ref{thm:genericity_general}.

\begin{thmx}\label{thmx:genericity}
Set
\begin{equation}\label{eq:Z_intro}
Z := \prod_{a \in \Sigma} f_a \in \mathcal{A}^W.
\end{equation}
This element is defined for every unramified $G$. If $\rho^\vee\in X$, then
$$e_{\mathrm{sgn}}\mathcal{H}e_{K} \,=\, \mathcal{A}^{W}\,\Omega^{-}e_{K},$$
where $\Omega^{-}$ is the image of $\Omega^{+}$ under $\theta_\lambda \mapsto \theta_{-\lambda}$, and $Z=\Omega^{-}\Omega^{+}$. For every unramified $G$, the irreducible spherical representation $V_G(\chi)$ with Satake parameter $\chi$ (Definition \ref{def:spherical_rep}) satisfies
$$\mathrm{Hom}_{\overline{U}}\big(V_G(\chi),\, \omega\big) \neq 0 \iff \chi(Z) \neq 0.$$
\end{thmx}

Theorem \ref{thmx:genericity} recovers the genericity condition of J.-S.\ Li \cite[Theorem 2.2]{Li92}, with Li's condition applied componentwise as explained in Remark \ref{rem:Li}. For split $G$, one has $Z = \prod_{\alpha \in \Phi}(1 - q\,\theta_{\alpha^\vee})$.

When $\rho^\vee \notin X$, the translation element $t_{\rho^\vee}$, the weights $\theta_{\pm\rho^\vee}$, and the factors $\Omega^{\pm}$ no longer exist, but the product form \eqref{eq:Z_intro} of $Z$ still does. Section \ref{sec:extended} removes this constraint. We use a central isogeny $p : G \to G''$, where $G''$ is unramified and its cocharacter lattice contains $\rho^\vee$. An embedding $j : \mathcal{H} \hookrightarrow \mathcal{H}''$ of Iwahori--Hecke algebras and restriction along $p$ transfer the results back to $G$ (Lemmas \ref{lem:extended_group}, \ref{lem:j_embedding}, and \ref{lem:transfer}).

\subsection*{The method}
The original proof of the Casselman--Shalika formula evaluates the Whittaker functional on Casselman's basis of Iwahori-fixed vectors in the unramified principal series, through functional equations with respect to the intertwining operators \cite{CS80}. Gurevich and Gurevich--Karasiewicz \cite{Gu13, GK22} derived this classical formula from a new source. Under the assumption $\rho^\vee \in X$, they proved that the $K$-invariants of the Gelfand--Graev representation $\mathrm{ind}_{\overline{U}}^{G}\,\overline{\psi}$, with $\overline{\psi}$ the conjugate of a character $\psi$ of conductor $\mathfrak{p}$, form a free $\mathcal{H}_{K}$-module of rank one. On the other hand, Chan and Savin \cite{CS18} (split groups) identified its $I$-invariants with the left ideal $\mathcal{H}e_{\mathrm{sgn}}$. We observe that the two results are related via the action of $e_{K}$ and the sign case of an identity of Brubaker, Bump, and Friedberg \cite[Theorem 5]{BBF16}. They arrive at the identity from the perspective that unique functionals on the universal unramified principal series are encoded by linear characters of the finite Hecke algebra. This paper recasts the algebraic core of \cite{BBF16} in terms of idempotents in the Iwahori--Hecke algebra. Moreover, the symmetric counterpart of the BBF identity, obtained by exchanging $e_K$ and $e_{\mathrm{sgn}}$, yields J.-S.\ Li's genericity criterion.

Our argument rests on the Bernstein presentation and the two free $\mathcal A$-modules of rank one $\mathcal{H}e_K = \mathcal{A}e_K$ and $\mathcal{H}e_{\mathrm{sgn}} = \mathcal{A}e_{\mathrm{sgn}}$. The second is further identified with the Iwahori invariants $(\mathrm{ind}_{\overline{U}}^{G}\,\overline{\psi})^I$ of the Gelfand--Graev representation. We study, in order, how $e_K\mathcal{H}e_K$, $e_K\mathcal{H}e_{\mathrm{sgn}}$, and $e_{\mathrm{sgn}}\mathcal{H}e_K$ sit inside $\mathcal{A}e_K$, $\mathcal{A}e_{\mathrm{sgn}}$, and $\mathcal{A}e_K$. The first computation applies to every unramified $G$ and recovers the Satake isomorphism. The Bernstein cross-relation identifies $e_K\mathcal H e_K$ inside $\mathcal A e_K$ with the $W$-invariant subspace $\mathcal A^W e_K$.

The other two computations assume $\rho^\vee\in X$. The computation of $e_K\mathcal{H}e_{\mathrm{sgn}}$ proves Theorem \ref{thmx:CS}, extending the sign case of the Brubaker--Bump--Friedberg (BBF) identity. It yields rank-one freeness first at conductor $\mathfrak p$ and then at conductor $\mathcal O$, with the Weyl characters appearing as the ratios $\mathrm{Alt}(\theta_{y+\rho^\vee})/\mathrm{Alt}(\theta_{\rho^\vee})$ in \eqref{eq:bbf_intro}. Proposition \ref{prop:twisted_satake} gives a Hecke-algebraic form of the twisted Satake transform of Gurevich and Gurevich--Karasiewicz \cite{Gu13, GK22}. From there we follow Gurevich--Karasiewicz closely, using the induction--compact induction pairing of Lemma \ref{lem:pairing}, which also appears in their work, to derive the Casselman--Shalika formula. Finally, the Iwahori--Matsumoto involution exchanges $e_K$ and $e_{\mathrm{sgn}}$, converting the preceding computation into $e_{\mathrm{sgn}}\mathcal H e_K=\mathcal A^W\Omega^-e_K$. This yields the genericity criterion for the spherical representation $V_G(\chi)$: it is generic exactly when $\chi(Z)\neq0$, where $Z=\prod_{a\in\Sigma}f_a=\Omega^-\Omega^+$ (Theorem \ref{thmx:genericity}). Section \ref{sec:extended} then passes to a central quotient whose cocharacter lattice contains $\rho^\vee$ and transfers the Casselman--Shalika formula as well as the genericity criterion to every unramified group.

\subsection*{Organization}

Section \ref{sec:notation} collects the preliminaries. Section \ref{sec:iwahori_fixed} identifies the two free $\mathcal{A}$-modules underlying the argument in this paper. Section \ref{sec:background} proves the Satake isomorphism via $e_K\mathcal{H}e_K$. Sections \ref{sec:bbf}--\ref{sec:cs} assume $\rho^\vee\in X$. Section \ref{sec:bbf} computes $e_K\mathcal H e_{\mathrm{sgn}}$ and proves the BBF identity. Section \ref{sec:ggp} proves rank-one freeness of the $K$-invariants of the Gelfand--Graev representation of conductor $\mathfrak p$, and Section \ref{sec:conductorO} transfers the result to conductor $\mathcal O$. Section \ref{sec:cs} proves the Casselman--Shalika formula and the genericity criterion. Section \ref{sec:extended} removes the hypothesis $\rho^\vee\in X$.

\section*{Acknowledgements}
The author thanks Professor Edmund Karasiewicz for explaining the IMRN paper he coauthored with Professor Nadya Gurevich and for sharing many insights from their work. The author is grateful to Professor Caihua Luo for inviting him to present this work at the Chinese University of Hong Kong, Shenzhen, in June 2024. The author thanks Professor Kei Yuen Chan and Professor Gordan Savin for inspiring conversations. The author was partially supported by a C.~R.~Wiley Instructorship at the University of Utah.

\section*{Tool and computational resource disclosure}

When preparing the paper, the human author used AI tools at two levels. At the model level, Claude Opus 4.8, Claude Fable 5, and GPT 5.6-Sol were used. At the agent level, a custom agent built by the author for mathematical research was used along with Claude Code and Codex. These tools assisted with mathematical development, exposition, citation checking, and proofreading. The human author verified all mathematical content and takes full responsibility for it.

\section{Preliminaries}\label{sec:notation}

\subsection{The field and the group}
Let $F/\mathbb{Q}_p$ be finite, with ring of integers $\mathcal{O}$, maximal ideal $\mathfrak{p}$, uniformizer $\varpi$, and residue field $k$ of cardinality $q$. Let $G/F$ be an unramified connected reductive group. Fix a maximal $F$-split torus $S \subseteq G$ and put $T := Z_G(S)$, a maximal torus. Fix an $F$-Borel $B = TU$ with unipotent radical $U$, and let $\overline{B} = T\overline{U}$ be the opposite Borel. Choose a hyperspecial vertex in the apartment of $S$. Let $G_{\mathcal O}$ be the associated reductive $\mathcal O$-model, and let $T_{\mathcal O}\subseteq G_{\mathcal O}$ be the unramified maximal torus with generic fiber $T$. Then
$$K := G_{\mathcal O}(\mathcal{O})$$
is a hyperspecial maximal compact subgroup of $G(F)$ and contains $T(\mathcal O):=T_{\mathcal O}(\mathcal O)$, the maximal compact subgroup of $T(F)$.
When no confusion can arise, we use the same symbols for these algebraic groups and their groups of $F$-points.

\subsection{Root data}
Let $\Phi := \Phi(G,S)$ be the relative root system, $\Phi^+$ the positive roots determined by $B$, $\Delta \subseteq \Phi^+$ the simple roots, and $W = N_G(S)(F)/T(F)$ the relative Weyl group with simple generators $\{s_\alpha\}_{\alpha\in\Delta}$ and length function $\ell$. The system $\Phi$ may fail to be reduced. A root $\alpha \in \Phi$ is \emph{multipliable} if $2\alpha \in \Phi$ and \emph{divisible} if $\tfrac{1}{2}\alpha \in \Phi$. Let
$$\Sigma := \{a \in \Phi : 2a \notin \Phi\}$$
be the set of non-multipliable roots. It is a reduced root system with the same Weyl group $W$ and positive system $\Sigma^+ = \Sigma \cap \Phi^+$. For a non-divisible $\beta \in \Phi$, we write
$$\tilde{\beta} := \begin{cases}\beta & \text{if } \beta \text{ is not multipliable},\\ 2\beta & \text{if } \beta \text{ is multipliable}\end{cases},$$
the unique element of $\Sigma$ on the ray of $\beta$. The simple roots of $\Sigma$ are the $\tilde\alpha$ with $\alpha \in \Delta$, and $s_{\tilde\alpha} = s_\alpha$. For unramified $G$, $\Sigma$ is the Bruhat--Tits \'echelonnage root system \cite[Theorem 6.3]{Haines18} (see also \cite[Section 4]{Tits}), and it agrees with Rostami's scaled root system, with root and coroot $\tilde\beta$ and $\tilde\beta^\vee$ on the ray of $\beta$ \cite[Section 3.3]{Rostami}.

Let $X := X_*(S)$ be the cocharacter lattice of $S$ and $X^*(S)$ the character lattice, with pairing $\langle\,\cdot\,,\,\cdot\,\rangle : X^*(S) \times X \to \mathbb{Z}$. Each $a \in \Sigma$ has a coroot $a^\vee \in X$. If $a$ is divisible, then $\langle a, X\rangle \subseteq 2\mathbb{Z}$, because $\tfrac{1}{2}a \in X^*(S)$. We set
$$m_a := \begin{cases} 1 & \text{if } a \in \Sigma \text{ is not divisible},\\ 2 & \text{if } a \in \Sigma \text{ is divisible}\end{cases},$$
a $W$-invariant function on $\Sigma$.

Let $E/F$ be the minimal splitting extension of $G$, an unramified extension whose Galois group is generated by the Frobenius $\sigma$. Write $X_*(T):=\operatorname{Hom}_{E\text{-grp}}(\mathbb G_{m,E},T_E)$, equipped with its Frobenius action. Let $\Phi_{\mathrm{abs}} := \Phi(G_E, T_E)$ be the absolute root system with the positive system $\Phi^+_{\mathrm{abs}}$ determined by $B_E$. Restriction from $T$ to $S$ carries $\Phi^+_{\mathrm{abs}}$ onto $\Phi^+$ and absolute simple roots to relative simple roots, and $X = X_*(T)^\sigma$ inside $X_*(T)$. Set
$$\rho^\vee := \tfrac{1}{2}\sum_{\beta\in\Phi^+_{\mathrm{abs}}}\beta^\vee \in X_*(T)\otimes\mathbb{Q}, \qquad 2\rho := \sum_{\beta\in\Phi^+_{\mathrm{abs}}}\beta|_S \in X^*(S).$$
Since $\sigma$ permutes $\Phi^+_{\mathrm{abs}}$, the element $\rho^\vee$ is $\sigma$-invariant and lies in $X\otimes\mathbb{Q}$. From $\langle\beta,\rho^\vee\rangle = 1$ for every absolute simple root $\beta$, we obtain
\begin{equation}\label{eq:rho_pairings}
\langle \alpha, \rho^\vee\rangle = 1 \quad (\alpha \in \Delta), \qquad \langle a, \rho^\vee\rangle = m_a \quad (a \in \Sigma \text{ simple}).
\end{equation}
Let $X^+$ and $X^{++}$ consist of the $\lambda\in X$ satisfying $\langle\alpha,\lambda\rangle\geq 0$ and $\langle\alpha,\lambda\rangle>0$, respectively, for every $\alpha\in\Delta$, and put $X^-:=-X^+$ and $X^{--}:=-X^{++}$. For $\lambda \in X$, write $t_\lambda := \lambda(\varpi^{-1}) \in T(F)$. The map $\lambda \mapsto t_\lambda$ identifies $X$ with $T(F)/T(\mathcal{O})$. The modular character of $B$ satisfies $\delta_B(t_\lambda) = q^{\langle 2\rho,\lambda\rangle}$ and $\delta_{\overline{B}} = \delta_B^{-1}$, and the Cartan and Iwasawa decompositions read
$$G = \bigsqcup_{\lambda \in X^+} K t_\lambda K, \qquad G = \bigsqcup_{\lambda \in X} \overline{U}\, t_\lambda\, K.$$

\subsection{Root groups, integral intersections, and the Iwahori subgroup}
For every relative root $\gamma\in\Phi$, let $U_\gamma$ be its relative root subgroup and put $d_\gamma:=\dim_FU_\gamma$. For every $\beta\in\Phi$, if $2\beta\notin\Phi$, put $U_{2\beta}=\{1\}$ and $d_{2\beta}=0$. If $\beta\in\Phi$ is non-divisible, then $U_{2\beta}$ is central in $U_\beta$ and $U_\beta/U_{2\beta}$ is abelian \cite[Lemma 3.3.8]{CGP}. Define the two Moy--Prasad intersections at the hyperspecial vertex by
$$U_{\beta,0}:=U_\beta(F)\cap K,\qquad U_{\beta,0+}:=U_\beta(F)\cap\ker\bigl(K\to G(k)\bigr).$$

Let $I\subseteq K$ be the preimage of $B(k)$ under reduction. The direct-spanning result \cite[Proposition 7.3.12]{KP} gives, in any fixed order of the non-divisible roots $\beta\in\Phi^+$,
$$U(F)\cap K=\prod_\beta U_{\beta,0},\qquad \overline U(F)\cap K=\prod_\beta U_{-\beta,0},\qquad \overline U(F)\cap I=\prod_\beta U_{-\beta,0+}.$$
Only as shorthand for these intersections, we now write
$$U(\mathcal O):=U(F)\cap K,\qquad \overline U(\mathcal O):=\overline U(F)\cap K,\qquad \overline U(\mathfrak p):=\overline U(F)\cap I.$$
The Iwahori factorization \cite[Corollary 7.4.9]{KP} is
$$I=U(\mathcal O)\,T(\mathcal O)\,\overline U(\mathfrak p).$$
Then $K=\bigsqcup_{w\in W}IwI$. Set
$$q_w := [IwI : I] \qquad (w \in \widetilde{W}),$$
where $\widetilde{W} := N_G(S)(F)/T(\mathcal{O}) \cong W \ltimes X$ is the Iwahori--Weyl group, with $\lambda \in X$ embedded as $t_\lambda$ and with length function $\ell$. Then $q_{uv} = q_u\, q_v$ whenever $\ell(uv) = \ell(u) + \ell(v)$, and for a simple reflection
$$q_{s_\alpha} = \bigl|U_{-\alpha,0}/U_{-\alpha,0+}\bigr| = q^{\,d_\alpha}, \qquad \alpha \in \Delta.$$
Since the simple parameters are constant on $W$-orbits, the rule $q_{\tilde\alpha}:=q_{s_\alpha}$ extends to a $W$-invariant function $a\mapsto q_a$ on $\Sigma$. We write $q_a=q^{d(a)}$. Define
$$W(q) := \sum_{w\in W}q_w = [K : I], \qquad W(q^{-1}) := \sum_{w\in W}q_w^{-1}.$$
We normalize the Haar measure on $G$ so that $\mathrm{vol}(I) = 1$. Consequently, $\mathrm{vol}(K) = W(q)$. For $\lambda\in X^{-}$,
$$[It_\lambda I : I] = q_{t_\lambda} = \prod_{\substack{\beta\in\Phi^+\\ \beta\text{ non-divisible}}}q^{-(d_\beta+d_{2\beta})\langle\beta,\lambda\rangle} = q^{-\langle 2\rho,\lambda\rangle} = \delta_{\overline B}(t_\lambda).$$
For the longest element $w_0 \in W$,
$$q_{w_0} = \prod_{a\in\Sigma^+}q_a = \big[\overline{U}(\mathcal{O}) : \overline{U}(\mathfrak{p})\big],$$
by the direct-spanning decomposition above.

\subsection{Convolution}
The algebra $C_c(G)$ of compactly supported locally constant functions on $G$ carries the convolution $(f_1 * f_2)(g) = \int_G f_1(h)\, f_2(h^{-1}g)\, dh$. Every smooth representation of $G$ is a $C_c(G)$-module through $\varphi \cdot v := \int_{G} \varphi(h)\, hv\, dh$. On functions on $G$ under right translation, $(\varphi\cdot f)(g) = \int_G \varphi(h)\,f(gh)\,dh$. With the geometric anti-involution $\check{\varphi}(g) := \varphi(g^{-1})$,
$$\varphi\cdot f \,=\, f * \check{\varphi}, \qquad \text{equivalently} \qquad f * \varphi \,=\, \check{\varphi}\cdot f.$$

\subsection{The Iwahori--Hecke algebra and the Bernstein presentation}\label{subsec:bernstein}
The Iwahori--Hecke algebra $\mathcal{H} = \mathcal{H}(G,I)=C_c(I\backslash G/I)$ has unit $\mathbf{1}_I$. For $w \in \widetilde{W}$, set $T_w := \mathbf{1}_{IwI}$. The set $\{T_w\}_{w\in\widetilde{W}}$ is a $\mathbb{C}$-basis of $\mathcal{H}$, and the Iwahori--Matsumoto relations hold \cite{IM65}, \cite[Proposition 4.1.1]{Rostami}:
$$T_u T_v = T_{uv} \quad \text{whenever } \ell(uv) = \ell(u) + \ell(v), \qquad T_{s_\alpha}^{2} = (q_{s_\alpha}-1)\,T_{s_\alpha} + q_{s_\alpha} \quad (\alpha \in \Delta).$$ With the idempotent $e_K$ defined in Section \ref{subsec:idempotents}, the spherical Hecke algebra is $\mathcal{H}_K := e_K \mathcal{H} e_K=C_c(K\backslash G/K)$.

The Bernstein elements $\{\theta_\lambda\}_{\lambda\in X}$ of $\mathcal{H}$ are defined by
$$\theta_\lambda = [I t_\lambda I : I]^{-1/2}\,\mathbf{1}_{I t_\lambda I} = q_{t_\lambda}^{-1/2}\,T_{t_\lambda} \qquad \text{for } \lambda \in X^-,$$
and extended to $X$ multiplicatively via $\theta_\lambda = \theta_{\lambda_1}\theta_{\lambda_2}^{-1}$ for any decomposition $\lambda = \lambda_1 - \lambda_2$ with $\lambda_i \in X^-$. Set $\mathcal{A} := \mathrm{span}_\mathbb{C}\{\theta_\lambda : \lambda \in X\} \cong \mathbb{C}[X]$, with $w(\theta_\lambda) = \theta_{w\lambda}$, and let $\mathcal{H}_0 := \mathcal{H}(K,I)$ be the finite Hecke algebra. For $\alpha \in \Delta$, with $a = \tilde\alpha \in \Sigma$, the Bernstein cross-relation is
\begin{equation}\label{eq:cross_relation}
T_{s_\alpha} F = s_\alpha(F)\,T_{s_\alpha} + \mathcal{G}_a\cdot\big(F - s_\alpha(F)\big), \qquad \mathcal{G}_a := \frac{q_a - 1 + \big(c_1(a) - c_2(a)\big)\,\theta_{a^\vee}}{1 - \theta_{2a^\vee}}, \qquad F\in\mathcal{A}.
\end{equation}
In the hyperplane notation of \cite[Section 7.1]{Rostami}, write $q(H)$ for the Bruhat--Tits parameter of an affine wall $H$. Then $q_a=q(\{a=0\})$, and we set $q_{a*}:=q(\{a=1\})$. The constants in \eqref{eq:cross_relation} are
$$c_1(a) := \big(q_a\, q_{a*}\big)^{1/2}, \qquad c_2(a) := \big(q_a/ q_{a*}\big)^{1/2}.$$
Wall parameters are constant on $\widetilde{W}$-orbits, and $c_1(a)c_2(a)=q_a$. If $a$ is not divisible, then $(c_1(a),c_2(a))=(q_a,1)$. If $a=2\alpha'$ is divisible, with $\alpha'\in\Delta$ multipliable, the rank-one parameter table gives
$$q_{a*}=q^{d_{2\alpha'}},\qquad d_{\alpha'}=3d_{2\alpha'},\qquad c_1(a)=q^{2d_{2\alpha'}},\qquad c_2(a)=q^{d_{2\alpha'}}.$$
The equality in the non-divisible case and the identities in the divisible case follow from \cite[Sections 1.8 and 4]{Tits}. In the divisible case, $\langle a,X\rangle\subseteq2\mathbb Z$, so $\mathcal G_a(F-s_\alpha(F))\in\mathcal A$.

Our $\theta_\lambda$ is Rostami's $\Theta_{-\lambda}$. The finite geometric expansion of $\mathcal{G}_a\big(F-s_\alpha(F)\big)$ turns \eqref{eq:cross_relation} into Rostami's relation \cite[Section 5.4]{Rostami}, with even and odd coefficients $q_a-1$ and $c_1(a)-c_2(a)$. This is the two-parameter relation of \cite[Proposition 3.6]{Lus89}.

\begin{fact}[Bernstein presentation]\label{fact:bernstein}
The elements $\theta_\lambda T_w$ with $\lambda \in X$ and $w \in W$ form a $\mathbb{C}$-basis of $\mathcal{H}$. Moreover, $\mathcal{H}$ is presented as the associative $\mathbb{C}$-algebra generated by the subalgebras $\mathcal{H}_0$ and $\mathcal{A} \cong \mathbb{C}[X]$ subject to their internal relations together with the cross-relation \eqref{eq:cross_relation}. In particular, an algebra homomorphism out of $\mathcal{H}$ may be specified by giving it on $\mathcal{H}_0$ and on $\mathcal{A}$ and checking that the images satisfy the quadratic, braid, lattice, and cross-relations.
\end{fact}

For $f \in \mathcal{A}^W$ and $\alpha \in \Delta$, the $\mathcal{G}_a$-term of the cross-relation \eqref{eq:cross_relation} vanishes, so $f$ commutes with each $T_{s_\alpha}$, hence with $\mathcal{H}_0$. Since $\mathcal{A}$ is commutative, Fact \ref{fact:bernstein} implies that $\mathcal{A}^W$ is contained in the center of $\mathcal{H}$. Although the arguments of this paper need only this inclusion, Rostami proves the following equality \cite{Rostami}.

\begin{fact}\label{fact:center}
The center of $\mathcal{H}$ equals $\mathcal{A}^W$.
\end{fact}

\subsection{Weyl characters}\label{subsec:chV}

Set
$$X'' := X + \mathbb{Z}\rho^\vee, \qquad \mathcal{A}'' := \mathbb{C}[X''],$$
with basis $\{\theta_\nu\}_{\nu\in X''}$ extending that of $\mathcal{A} \cong \mathbb{C}[X]$ and with the $W$-action $w(\theta_\nu) = \theta_{w\nu}$. By \eqref{eq:rho_pairings}, $s_a(\rho^\vee) = \rho^\vee - m_a a^\vee$ for every simple root $a$ of $\Sigma$, so $W$ preserves $X''$. The inclusion $\langle a,X\rangle\subseteq m_a\mathbb{Z}$, together with \eqref{eq:rho_pairings} and $W$-stability, then implies $\langle a, X''\rangle \subseteq m_a\mathbb{Z}$ for every $a \in \Sigma$. The dominant cone $(X'')^+$ is defined as for $X$, and $X'' = X$ whenever $\rho^\vee \in X$.

Define $\mathrm{Alt}(F) := \sum_{w\in W}(-1)^{\ell(w)}w(F)$ for $F \in \mathcal{A}''$, and let $\Sigma_m^\vee := \{m_a a^\vee : a \in \Sigma\}$. This is a reduced root system with coroots $\{a/m_a : a \in \Sigma\}$, equivalently the coroot system of the non-divisible roots in $\Phi$. The reflection corresponding to $m_a a^\vee$ is $s_a$, so its Weyl group is $W$, and $X''$ and $X$ are character lattices for $\Sigma_m^\vee$. With $\varrho := \tfrac12\sum_{a\in\Sigma^+}m_a a^\vee$, \eqref{eq:rho_pairings} shows that $\rho^\vee-\varrho$ is $W$-invariant. Thus, by \cite[Chapter VI, \S 3, no.~3, Proposition~2(i)]{Bou}, we have the following identity.

\begin{fact}[Weyl denominator formula]\label{fact:denominator}
$$\mathrm{Alt}\big(\theta_{\rho^\vee}\big) = \theta_{\rho^\vee}\prod_{a\in\Sigma^+}\big(1 - \theta_{-m_a a^\vee}\big).$$
\end{fact}

By Fact \ref{fact:denominator}, the next fact is the Weyl character formula for $\Sigma_m^\vee$, with character lattices $X''$ and $X$. See \cite[Chapter VI, \S 3, no.~3, Propositions~1 and~2(ii)--(iii)]{Bou} for part (1) and \cite[Chapter VI, \S 3, no.~3, Remark~2, and no.~4, Proposition~3]{Bou} for part (2).

\begin{fact}[Weyl character formula]\label{fact:chV}
\begin{enumerate}
\item For $\nu \in (X'')^+$, the element
\[ \mathrm{ch}V_\nu := \frac{\mathrm{Alt}(\theta_{\nu+\rho^\vee})}{\mathrm{Alt}(\theta_{\rho^\vee})} \]
is a well-defined element of $(\mathcal{A}'')^W$, and $\{\mathrm{ch}V_\nu : \nu \in (X'')^+\}$ is a $\mathbb{C}$-basis of $(\mathcal{A}'')^W$.
\item For $y \in X^+$, the element $\mathrm{ch}V_y$ belongs to $\mathcal{A}^W$, and $\{\mathrm{ch}V_y : y \in X^+\}$ is a $\mathbb{C}$-basis of $\mathcal{A}^W$. If $F\in\mathcal{A}^W$ has the expansion $F=\sum_{\nu\in(X'')^+}c_\nu\,\mathrm{ch}V_\nu$ in the basis of (1), then $c_\nu=0$ for every $\nu\notin X^+$.
\end{enumerate}
\end{fact}

The final assertion of (2) holds because $\{\mathrm{ch}V_y : y\in X^+\}$ is a subfamily of the basis in (1), so the two expansions of $F$ coincide.

\begin{remark}\label{rem:chV_interpretation}
Fix $x\in X^+$. When $G$ is split, $\mathrm{ch}V_x$ is the character of the irreducible $G^\vee$-representation of highest weight $x$. For general unramified $G$, it may be interpreted either as the character of the irreducible ${}^{L}G^\dagger$-representation of highest weight $x$, through the inclusion of character rings in \cite[Section 2.9 and Lemma 2.11]{GK22}, or as the normalized Frobenius-twining character of the irreducible $G^\vee$-representation of highest weight $x$ \cite[Theorem 7.4]{Hopper}. Neither interpretation is used below.
\end{remark}

\subsection{Idempotents}\label{subsec:idempotents}
Inside $\mathcal{H}_0$ we consider the two elements
$$e_K = \frac{1}{W(q)}\sum_{w\in W} T_w = \frac{1}{W(q)}\,\mathbf{1}_K, \qquad e_{\mathrm{sgn}} = \frac{1}{W(q^{-1})}\sum_{w\in W} (-1)^{\ell(w)}\,q_w^{-1}\, T_w.$$

\begin{lemma}\label{lem:idempotents}
Both $e_K$ and $e_{\mathrm{sgn}}$ are idempotents. For every $\alpha \in \Delta$, we have $T_{s_\alpha} e_K = e_K T_{s_\alpha} = q_{s_\alpha}\, e_K$ and $T_{s_\alpha} e_{\mathrm{sgn}} = e_{\mathrm{sgn}} T_{s_\alpha} = -e_{\mathrm{sgn}}$. Moreover, any $h \in \mathcal{H}_0$ satisfying $T_{s_\alpha} h = q_{s_\alpha}h$ for all $\alpha \in \Delta$ lies in $\mathbb{C}e_K$, and any $h \in \mathcal{H}_0$ satisfying $T_{s_\alpha} h = -h$ for all $\alpha \in \Delta$ lies in $\mathbb{C}e_{\mathrm{sgn}}$.
\end{lemma}
\begin{proof}
Fix $\alpha \in \Delta$ and write $s = s_\alpha$. Partition $W$ into pairs $\{w, sw\}$ with $\ell(sw) = \ell(w)+1$. Such a pair contributes $(1+T_s)T_w$ to $W(q)\,e_K$ and $(-1)^{\ell(w)}q_w^{-1}(1 - q_s^{-1}T_s)T_w$ to $W(q^{-1})\,e_{\mathrm{sgn}}$. Then the quadratic relation implies $T_s(1+T_s) = q_s(1+T_s)$ and $T_s(1 - q_s^{-1}T_s) = -(1 - q_s^{-1}T_s)$. Hence $T_s e_K = q_s\,e_K$ and $T_s e_{\mathrm{sgn}} = -e_{\mathrm{sgn}}$. The partition into pairs $\{w, ws\}$ proves the right-hand identities. Moreover, iteration shows that $e_K T_w = q_w\, e_K$ and $e_{\mathrm{sgn}} T_w = (-1)^{\ell(w)} e_{\mathrm{sgn}}$ for all $w \in W$, and the chosen normalizations then imply $e_K^2 = e_K$ and $e_{\mathrm{sgn}}^2 = e_{\mathrm{sgn}}$.

For the last statement, if $T_{s_\alpha}h = q_{s_\alpha}h$ for all $\alpha\in\Delta$, then $T_w h = q_wh$ for all $w \in W$, so $h = e_K h \in e_K \mathcal{H}_0 = \mathbb{C}e_K$. The sign case is identical.
\end{proof}

\subsection{The Iwahori--Matsumoto involution}\label{subsec:IM}

The next lemma constructs the Iwahori--Matsumoto involution in our unequal-parameter setting, cf.\ \cite[Section 2.2]{Jan}.

\begin{lemma}\label{lem:IM}
There is a unique algebra automorphism $\iota$ of $\mathcal{H}$ with
$$\iota(T_{s_\alpha}) = (q_{s_\alpha}-1) - T_{s_\alpha} = -q_{s_\alpha}\,T_{s_\alpha}^{-1} \quad (\alpha \in \Delta), \qquad \iota(\theta_\lambda) = \theta_{-\lambda} \quad (\lambda \in X).$$
It is an involution and satisfies $\iota(e_K) = e_{\mathrm{sgn}}$ and $\iota(e_{\mathrm{sgn}}) = e_K$.
\end{lemma}
\begin{proof}
By Fact \ref{fact:bernstein}, it suffices to check that the assignments preserve the defining relations. The braid relations are preserved under $T_{s} \mapsto -q_sT_{s}^{-1}$. Indeed, the inverses satisfy them, and the accumulated scalars on the two sides agree, since each generator occurs equally often when the braid length is even, while for odd braid length the two reflections are conjugate in $W$ and their parameters coincide. For the quadratic relation, with $q_s = q_{s_\alpha}$,
$$\big((q_s-1) - T_s\big)^2 = (q_s-1)^2 - 2(q_s-1)T_s + (q_s-1)T_s + q_s = (q_s-1)\big((q_s-1) - T_s\big) + q_s.$$
The assignment $\theta_\lambda \mapsto \theta_{-\lambda}$ induces the inversion automorphism of $\mathcal{A} \cong \mathbb{C}[X]$, and $s_\alpha(\bar{F}) = \overline{s_\alpha(F)}$. For the cross-relation \eqref{eq:cross_relation}, writing $\bar{F}$ for $\iota(F)$ and $a = \tilde\alpha$,
\begin{align*}
\iota(T_{s_\alpha})\bar{F} - s_\alpha(\bar{F})\,\iota(T_{s_\alpha})
&= \big((q_a-1) - T_{s_\alpha}\big)\bar{F} - s_\alpha(\bar{F})\big((q_a-1) - T_{s_\alpha}\big) \\
&= (q_a-1)\big(\bar{F} - s_\alpha(\bar{F})\big) - \mathcal{G}_a\cdot\big(\bar{F} - s_\alpha(\bar{F})\big)
= s_\alpha\big(\mathcal{G}_a\big)\cdot\big(\bar{F} - s_\alpha(\bar{F})\big),
\end{align*}
using the identity $\mathcal{G}_a + s_\alpha(\mathcal{G}_a) = q_a - 1$, which is verified below after putting the terms over a common denominator, with $c_i:=c_i(a)$:
$$\mathcal{G}_a + s_\alpha(\mathcal{G}_a) = \frac{q_a - 1 + (c_1-c_2)\theta_{a^\vee}}{1-\theta_{2a^\vee}} - \frac{\big(q_a - 1\big)\theta_{2a^\vee} + (c_1-c_2)\theta_{a^\vee}}{1-\theta_{2a^\vee}} = q_a - 1.$$
Since $s_\alpha(\mathcal{G}_a)$ is $\mathcal{G}_a$ with $\theta_{a^\vee}$ replaced by $\iota(\theta_{a^\vee}) = \theta_{-a^\vee}$, this is the cross-relation transported by $\iota$. On the generators, $\iota^2 = \mathrm{id}$, so $\iota$ is an involution.

Since $\iota$ is an algebra map, $\iota(e_K)$ is a non-zero idempotent of $\mathcal{H}_0$, and for every $\alpha \in \Delta$,
$$T_{s_\alpha}\,\iota(e_K) = \iota\big(\iota(T_{s_\alpha})\, e_K\big) = \iota\big(((q_a-1) - T_{s_\alpha})e_K\big) = \iota(-e_K) = -\iota(e_K).$$
By Lemma \ref{lem:idempotents}, $\iota(e_K) \in \mathbb{C}e_{\mathrm{sgn}}$, and since it is a non-zero idempotent it equals $e_{\mathrm{sgn}}$. Applying $\iota$ again shows that $\iota(e_{\mathrm{sgn}}) = e_K$.
\end{proof}

\subsection{Whittaker characters and Gelfand--Graev representations}
Fix a $\sigma$-equivariant Chevalley--Steinberg system for $G_{\mathcal O_E}:=G_{\mathcal O}\otimes_{\mathcal O}\mathcal O_E$ \cite[Definition 2.9.15]{KP}. For each non-divisible $\beta$, let $E_\beta/F$ be the corresponding root field, an unramified subextension of $E/F$. The associated relative-root coordinate \cite[Construction 2.9.16]{KP} is an isomorphism of additive groups
$$x_{[\beta]}:(E_\beta,+)\xrightarrow{\ \sim\ }(U_\beta/U_{2\beta})(F).$$
By \cite[Lemma 9.6.3]{KP}, the normalization from the integral model identifies the hyperspecial filtrations as
$$x_{[\beta]}(\mathcal O_{E_\beta})=\operatorname{im}(U_{\beta,0}),\qquad x_{[\beta]}(\mathfrak p_{E_\beta})=\operatorname{im}(U_{\beta,0+}),$$
and the $\beta$-weight action of $T$ gives
\begin{equation}\label{eq:root_coordinate_shift}
t_\lambda x_{[\beta]}(c)t_\lambda^{-1}=x_{[\beta]}\bigl(\varpi^{-\langle\beta,\lambda\rangle}c\bigr),\qquad c\in E_\beta,\ \lambda\in X.
\end{equation}

In view of the isomorphism of $F$-groups
$$\overline U/[\overline U,\overline U]\cong\prod_{\alpha\in\Delta}U_{-\alpha}/U_{-2\alpha},$$
a smooth character $\tau : \overline{U} \to \mathbb{C}^\times$ is determined by the characters
$$\tau_\alpha := \tau \circ x_{[-\alpha]} : E_\alpha=E_{-\alpha} \longrightarrow \mathbb{C}^\times, \qquad \alpha \in \Delta.$$
We say that $\tau$ has \emph{conductor $\mathfrak p$} if it is trivial on $\overline U(\mathfrak p)$ and every $\tau_\alpha$ is non-trivial on $\mathcal O_{E_\alpha}$. It has \emph{conductor $\mathcal O$} if it is trivial on $\overline U(\mathcal O)$ and every $\tau_\alpha$ is non-trivial on $\mathfrak p_{E_\alpha}^{-1}$. The values of $\tau$ are roots of unity, so the complex conjugate $\overline{\tau}$ equals $\tau^{-1}$ and has the same conductor.

For a smooth character $\tau$ of $\overline{U}$, the smooth induction $\mathrm{Ind}_{\overline{U}}^{G}\,\tau$ is the space of functions $f : G \to \mathbb{C}$ with
$$f(ug) = \tau(u)\,f(g) \qquad (u \in \overline{U},\ g \in G)$$
that are invariant under right translation by some open compact subgroup. The group $G$ acts by right translation, $(g \cdot f)(x) = f(xg)$, and $\mathrm{ind}_{\overline{U}}^{G}\,\tau$ is the subrepresentation of functions with compact support modulo $\overline{U}$. For a character $\tau$ of either conductor, the associated \emph{Gelfand--Graev representation} is the compact induction $\mathrm{ind}_{\overline{U}}^{G}\,\tau$. We denote characters of conductor $\mathfrak p$ and $\mathcal O$ by $\psi$ and $\omega$, respectively.

Gurevich--Karasiewicz prove the assertions of the next lemma for conductor $\mathfrak p$ \cite[Lemmas 5.1 and 5.2]{GK22}. We include a proof that treats both conductors uniformly.

\begin{lemma}\label{lem:whittaker_support}
Let $\tau$ be a character of $\overline U$ of conductor $\mathcal O$ or $\mathfrak p$, and put
$$X(\tau):=\begin{cases}X^+&\text{if $\tau$ has conductor $\mathcal O$},\\ X^{++}&\text{if $\tau$ has conductor $\mathfrak p$}\end{cases}.$$
\begin{enumerate}
\item Every $f\in(\mathrm{Ind}_{\overline U}^G\tau)^K$ satisfies $f(t_y)=0$ unless $y\in X(\tau)$.
\item For $y\in X(\tau)$, there is a well-defined function $\Phi^\tau_y\in(\mathrm{ind}_{\overline U}^G\tau)^K$ supported on $\overline U t_yK$ and determined by
$$\Phi^\tau_y(u t_yk)=\delta_{\overline B}(t_y)^{1/2}\tau(u)\qquad(u\in\overline U,\ k\in K).$$
The set $\{\Phi^\tau_y\mid y\in X(\tau)\}$ is a basis of $(\mathrm{ind}_{\overline U}^G\tau)^K$.
\end{enumerate}
\end{lemma}
\begin{proof}
For (1), first suppose that $\tau$ has conductor $\mathcal O$ and that $\langle\alpha,y\rangle\leq-1$ for some $\alpha\in\Delta$. For $u\in U_{-\alpha,0}\subseteq K$, we have
$$f(t_y)=f(t_yu)=\tau(t_yut_y^{-1})f(t_y).$$
As $u$ varies, its image in $U_{-\alpha}/U_{-2\alpha}$ runs through $x_{[-\alpha]}(\mathcal O_{E_\alpha})$. By \eqref{eq:root_coordinate_shift}, the image of $t_yut_y^{-1}$ runs through $x_{[-\alpha]}(\mathfrak p_{E_\alpha}^{\langle\alpha,y\rangle})$, which contains $x_{[-\alpha]}(\mathfrak p_{E_\alpha}^{-1})$. Since $\tau_\alpha$ is non-trivial on $\mathfrak p_{E_\alpha}^{-1}$, some $u$ satisfies $\tau(t_yut_y^{-1})\ne1$, and hence $f(t_y)=0$. If $\tau$ has conductor $\mathfrak p$ and $\langle\alpha,y\rangle\leq0$, the same argument applies because the coordinate range contains $\mathcal O_{E_\alpha}$, on which $\tau_\alpha$ is non-trivial. This proves (1).

For $y\in X(\tau)$, direct spanning provides
$$\overline U\cap t_yKt_y^{-1}=t_y\overline U(\mathcal O)t_y^{-1}.$$
By \eqref{eq:root_coordinate_shift}, the image of this group in each simple-root quotient $U_{-\alpha}/U_{-2\alpha}$ lies in $x_{[-\alpha]}(\mathcal O_{E_\alpha})$ when $\tau$ has conductor $\mathcal O$, and in $x_{[-\alpha]}(\mathfrak p_{E_\alpha})$ when $\tau$ has conductor $\mathfrak p$. Since $\tau$ factors through the product of these quotients and is trivial on the indicated lattices, it is trivial on $\overline U\cap t_yKt_y^{-1}$. Thus $\Phi^\tau_y$ is well defined. By the Iwasawa decomposition, any $f\in(\mathrm{ind}_{\overline U}^G\tau)^K$ is determined by the values $f(t_y)$. Compactness of the support modulo $\overline U$ leaves only finitely many non-zero values. Finally, the functions $\Phi^\tau_y$ are linearly independent because their supports are disjoint.
\end{proof}

\section{Iwahori Invariants and Two Free Modules}\label{sec:iwahori_fixed}

This section assembles the Iwahori-fixed inputs of the paper. Regarding the following fact, the admissible case goes back to Borel \cite[Theorem 4.10 and Corollaries 4.11--4.12]{Bo}. For unramified $G$, $(I,\mathbf 1)$ is a type for the unramified principal-series Bernstein component (see \cite[Proposition 3.3.1]{Haines12}). Bushnell and Kutzko \cite[Theorem 4.3(ii)--(iii)]{BK} then give the equivalence and decomposition for arbitrary smooth representations.

\begin{fact}\label{fact:borel}
The pair $(I,\mathbf 1)$ is a type for a single Bernstein component, the block of the unramified principal series. Let $\mathcal R_I(G)$ be this Bernstein direct summand of the category of smooth representations. Equivalently, it is the category of smooth representations generated by their $I$-fixed vectors. The functor
$$V\longmapsto V^I$$
is an equivalence from $\mathcal R_I(G)$ to the category of left $\mathcal H$-modules. In particular, it gives a bijection between irreducible smooth representations with non-zero $I$-fixed vectors and simple left $\mathcal H$-modules.

Let $p_I$ denote the Bernstein projection onto this component. For every smooth representation $\pi$ and every $V\in\mathcal R_I(G)$, Bernstein orthogonality and the type equivalence identify
$$\operatorname{Hom}_G(\pi,V)=\operatorname{Hom}_G(p_I\pi,V)\cong\operatorname{Hom}_{\mathcal H}\bigl((p_I\pi)^I,V^I\bigr)=\operatorname{Hom}_{\mathcal H}(\pi^I,V^I).$$
\end{fact}

Following Savin's \emph{tale of two Hecke algebras} \cite{Sav}, let $\mathcal{H}_{IK} := C_c(I\backslash G/K)$ be the $(\mathcal{H}, \mathcal{H}_K)$-bimodule on which $\mathcal{H}$ acts by left convolution. Then $\mathcal{H}_{IK} = \mathcal{H}e_K$. Moreover, for split groups, Savin shows that $\mathcal{H}e_K = \mathcal{A}e_K$ is free over $\mathcal{A}$ of rank one. The analogous statement holds for $e_{\mathrm{sgn}}$. We summarize these statements in the following lemma.

\begin{lemma}[Free ideals of rank one]\label{lem:HIK}
The maps $F \mapsto Fe_K$ and $F \mapsto Fe_{\mathrm{sgn}}$ are injective on $\mathcal{A}$, and
$$\mathcal{H}_{IK} = \mathcal{H}e_K = \mathcal{A}e_K, \qquad \mathcal{H}e_{\mathrm{sgn}} = \mathcal{A}e_{\mathrm{sgn}},$$
so both left ideals are free left $\mathcal{A}$-modules of rank one, generated by the respective idempotents.
\end{lemma}
\begin{proof}
Averaging over $K$ shows $f * e_K = f$ for every $f \in \mathcal{H}_{IK}$, so $\mathcal{H}_{IK} = \mathcal{H}e_K$. By Lemma \ref{lem:idempotents}, $\mathcal{H}_0e_K = \mathbb{C}e_K$ and $\mathcal{H}_0e_{\mathrm{sgn}} = \mathbb{C}e_{\mathrm{sgn}}$, so the Bernstein decomposition $\mathcal{H} = \mathcal{A}\cdot\mathcal{H}_0$ in Fact \ref{fact:bernstein} implies $\mathcal{H}e_K = \mathcal{A}e_K$ and $\mathcal{H}e_{\mathrm{sgn}} = \mathcal{A}e_{\mathrm{sgn}}$. For injectivity, expand $Fe_K$ and $Fe_{\mathrm{sgn}}$ in the basis $\{\theta_\lambda T_w\}$ defined in Fact \ref{fact:bernstein}: for $F \in \mathcal{A}$, the element $Fe_K = W(q)^{-1}\sum_{w\in W}F\,T_w$ has $T_e$-coordinate $W(q)^{-1}F$, and $Fe_{\mathrm{sgn}}$ has $T_e$-coordinate $W(q^{-1})^{-1}F$.
\end{proof}

Applying $e_K$ to these two free modules produces $e_K\mathcal{H}e_K$ and $e_K\mathcal{H}e_{\mathrm{sgn}}$, which Sections \ref{sec:background} and \ref{sec:bbf} study. The inversion $f \mapsto \check{f}$ identifies $\mathcal{H}_{IK}$ with $(\mathrm{ind}_K^G \mathbf{1})^{I}$, the Iwahori invariants of the spherical model. It is isomorphic to $\mathcal{H}e_{K}$ as an $\mathcal{H}$-module and free over $\mathcal{A}$ with an explicit generator $e_{K} \in \mathcal{H}_{IK}$. In parallel, for split groups, Chan and Savin \cite{CS18} identify the Iwahori invariants of the Gelfand--Graev representation with $\mathcal{H}e_{\mathrm{sgn}}$ as an $\mathcal{H}$-module and exhibit an explicit $\mathcal{A}$-module generator.

The explicit construction and calculation in \cite{CS18} are generalized to all connected reductive groups in \cite{LuoGCS}. We use the main result of \cite{LuoGCS} for the opposite Borel $\overline B$. Let $\psi$ be a character of $\overline{U}$ of conductor $\mathfrak{p}$. Since $\psi$ is trivial on $\overline{U} \cap I = \overline{U}(\mathfrak{p})$, there is a unique element
$$\mathrm{ch}_{I}^{\psi} \in \big(\mathrm{ind}_{\overline{U}}^{G}\,\psi\big)^I \quad\text{supported on } \overline{U}\cdot I, \qquad \mathrm{ch}_{I}^{\psi}(u\, i) = \psi(u) \quad (u\in\overline{U},\ i\in I).$$
Since $\overline{\psi}$ has the same conductor, this also defines $\mathrm{ch}_I^{\overline{\psi}}$.

\begin{fact}\label{fact:CS}
Let $\psi$ be a character of $\overline{U}$ of conductor $\mathfrak{p}$. Then the following hold.
\begin{enumerate}
\item $T_w \cdot \mathrm{ch}_I^{\psi} = (-1)^{\ell(w)}\, \mathrm{ch}_I^{\psi}$ for all $w \in W$. Equivalently, $e_{\mathrm{sgn}}\cdot \mathrm{ch}_I^{\psi} = \mathrm{ch}_I^{\psi}$.
\item The map $F \mapsto F\cdot \mathrm{ch}_I^{\psi}$ is an isomorphism of left $\mathcal{A}$-modules from $\mathcal{A}$ onto $(\mathrm{ind}_{\overline{U}}^G \psi)^I$. Consequently, $F e_{\mathrm{sgn}} \mapsto F \cdot \mathrm{ch}_I^{\psi}$ is an isomorphism of left $\mathcal{H}$-modules from $\mathcal{H}e_{\mathrm{sgn}} = \mathcal{A}e_{\mathrm{sgn}}$ onto $(\mathrm{ind}_{\overline{U}}^G \psi)^I$.
\end{enumerate}
\end{fact}

For more general Bernstein blocks, Solleveld \cite{Sol25} and Solleveld--Opdam \cite{SO26} identify the Hom-space from a Bernstein progenerator to the Gelfand--Graev representation with the module induced from the Steinberg character of the finite Hecke algebra. Fact \ref{fact:CS}(2) instead exhibits the generator for the Iwahori block, which is the form the computations of this paper require.

\section{The Satake Isomorphism and Spherical Representations}\label{sec:background}

This section carries out the first computation announced in the introduction: the explicit description of $e_K\mathcal{H}e_K$ inside the free module $\mathcal{H}e_K = \mathcal{A}e_K$ identified in Lemma \ref{lem:HIK}.

The classical Satake isomorphism \cite{Sat63} establishes that the spherical Hecke algebra $\mathcal{H}_K = e_K \mathcal{H} e_K$ is isomorphic to $\mathcal{A}^W$. Knop \cite{Kno05} proves the corresponding spherical-subalgebra isomorphism for affine Hecke algebras with unequal parameters. The proof in \cite{Sav} for split groups runs through the center of $\mathcal{H}$ (Fact \ref{fact:center}). Here we record a short direct proof from the cross-relation \eqref{eq:cross_relation}, without using the center. The same projection technique recurs in Section \ref{sec:bbf}.

\begin{theorem}[Satake isomorphism] \label{thm:spherical_iso}
For $F \in \mathcal{A}$, the element $F e_K$ belongs to $e_K \mathcal{H} e_K$ if and only if $s_\alpha(F) = F$ for all simple roots $\alpha$. Consequently, the spherical Hecke algebra $\mathcal{H}_K = e_K \mathcal{H} e_K$ is isomorphic to $\mathcal{A}^W$ as a $\mathbb{C}$-algebra via the map $F \mapsto F e_K$.
\end{theorem}
\begin{proof}
By Lemma \ref{lem:HIK}, every element of $\mathcal{H} e_K$ is of the form $F e_K$ for a unique $F \in \mathcal{A}$.

By Lemma \ref{lem:idempotents}, $F e_K$ lies in $e_K \mathcal{H} e_K$ if and only if $T_{s_\alpha} (F e_K) = q_{s_\alpha} (F e_K)$ for all $\alpha \in \Delta$.

Fix $\alpha \in \Delta$ and write $s = s_\alpha$ and $a = \tilde\alpha$. The cross-relation \eqref{eq:cross_relation} and $T_{s} e_K = q_a e_K$ lead to
\begin{equation*}
T_{s} (F e_K) = q_a\, s(F)\, e_K + \mathcal{G}_a\cdot\big(F - s(F)\big)\, e_K.
\end{equation*}
Equating this to $q_a F e_K$ and using the injectivity in Lemma \ref{lem:HIK} to isolate the $\mathcal{A}$-coordinates yields
\begin{equation*}
\big(F - s(F)\big)\,\big(\mathcal{G}_a - q_a\big) = 0
\end{equation*}
in $\mathrm{Frac}(\mathcal{A})$. Clearing the denominator $1 - \theta_{2a^\vee}$ of $\mathcal{G}_a$, the second factor becomes
\begin{equation*}
q_a - 1 + \big(c_1(a) - c_2(a)\big)\theta_{a^\vee} - q_a\big(1 - \theta_{2a^\vee}\big) = q_a\theta_{2a^\vee} + \big(c_1(a)-c_2(a)\big)\theta_{a^\vee} - 1,
\end{equation*}
which is non-zero in $\mathcal{A}$ because its $\theta_0$-coefficient is $-1$. Since $\mathcal{A}$ is an integral domain, $s(F) = F$. This holds for all $\alpha \in \Delta$, so $F \in \mathcal{A}^W$. Conversely, the same computation shows that $F \in \mathcal{A}^W$ implies $T_{s_\alpha}(Fe_K) = q_{s_\alpha} F e_K$ for all $\alpha \in \Delta$, so $F \mapsto F e_K$ is a linear isomorphism $\mathcal{A}^W \xrightarrow{\sim} e_K \mathcal{H} e_K$.

For $F, F' \in \mathcal{A}^W$ we have $e_K (F' e_K) = F' e_K$, so $(F e_K)(F' e_K) = F \big(e_K F' e_K\big) = F F' e_K$, and $1 \cdot e_K = e_K$ is the unit of $\mathcal{H}_K$. Hence $F \mapsto F e_K$ is an isomorphism of $\mathbb{C}$-algebras $\mathcal{A}^W \cong \mathcal{H}_K$.
\end{proof}

\begin{remark}
Macdonald's formula for the zonal spherical function can be proved by the projection method of Section \ref{sec:bbf}. We do not pursue this here.
\end{remark}

As an application, we classify the simple $\mathcal{H}$-modules $M$ with $e_K M \neq 0$.

\begin{lemma}[The spherical module]\label{lem:spherical_unique}
Let $\chi : \mathcal{H}_K \to \mathbb{C}$ be an algebra character. Up to isomorphism there is a unique simple left $\mathcal{H}$-module $M$ with $e_K M \neq 0$ such that $\mathcal{H}_K$ acts on $e_K M$ through $\chi$. For this module, $\dim_{\mathbb{C}} e_K M = 1$. We denote it $M_G(\chi)$.
\end{lemma}
\begin{proof}
By Fact \ref{fact:borel}, the simple $\mathcal{H}$-modules with $e_K M \neq 0$ are the $M = V^I$ for irreducible smooth representations $V$ with $V^K = e_KV^I \neq 0$. Applying \cite[Proposition D.1.8]{Lau96} to the compact open subgroup $K$ gives a bijection $V \mapsto V^K = e_K M$ onto the isomorphism classes of simple $\mathcal{H}_K$-modules. Since $\mathcal{H}_K \cong \mathcal{A}^W$ is a finitely generated commutative $\mathbb{C}$-algebra, the Nullstellensatz shows that its simple modules are one-dimensional, given by the characters $\chi$.
\end{proof}

\begin{definition}\label{def:spherical_rep}
Let $V_G(\chi)$ denote the irreducible smooth representation of $G$ with $V_G(\chi)^I \cong M_G(\chi)$, provided by Fact \ref{fact:borel}. The space $V_G(\chi)^K = e_K M_G(\chi)$ is one-dimensional, spanned by a vector $v_K$ on which $\mathcal{H}_K$ acts through $\chi$. We call $V_G(\chi)$ the spherical representation with Satake parameter $\chi$.
\end{definition}

The following compatibility is expressed in the notation of the dominant and anti-dominant Bernstein presentations in \cite{HP02}. We include a short proof to fix our conventions and cover the present parameters.

\begin{proposition}[Geometric anti-involution on $\mathcal{A}^W$] \label{prop:anti_involution_center}
Recall from Section \ref{sec:notation} the geometric anti-involution $h \mapsto \check{h}$ of $\mathcal{H}$, $\check{h}(g) = h(g^{-1})$. For $f \in \mathcal{A}$, let $\bar{f}$ denote the image of $f$ under $\theta_\lambda \mapsto \theta_{-\lambda}$. Then
\begin{equation*}
\check{f} = \bar{f} \qquad \text{for every } f \in \mathcal{A}^W.
\end{equation*}
\end{proposition}
\begin{proof}
For $\lambda \in X^-$, we have $\theta_\lambda = q_{t_\lambda}^{-1/2} T_{t_\lambda}$, hence $\check{\theta}_\lambda = q_{t_\lambda}^{-1/2} T_{t_{-\lambda}}$. After translating Rostami's dominant-translation convention to our choice $t_\lambda=\lambda(\varpi^{-1})$, the length formula \cite[Lemma 5.2.1(iii)--(iv)]{Rostami} shows that both products in $w_0t_{-\lambda} = t_{-w_0\lambda}w_0$ are length-additive. Therefore
\begin{equation*}
T_{w_0} T_{t_{-\lambda}} = T_{t_{-w_0\lambda}} T_{w_0}.
\end{equation*}
Since $w_0(2\rho) = -2\rho$, the parameters satisfy $q_{t_{-w_0\lambda}} = q_{t_\lambda}$, and therefore
\begin{equation}\label{eq:check_theta}
\check{\theta}_\lambda = T_{w_0}^{-1}\theta_{-w_0\lambda}T_{w_0}.
\end{equation}
For general $\lambda = \lambda_1 - \lambda_2$ with $\lambda_i \in X^-$, $\check{\theta}_\lambda = (\check{\theta}_{\lambda_2})^{-1}\check{\theta}_{\lambda_1}$, so \eqref{eq:check_theta} holds for all $\lambda \in X$. By linearity, $\check{f} = T_{w_0}^{-1}\, f^{\sharp}\, T_{w_0}$ for every $f \in \mathcal{A}$, where $f^{\sharp}$ denotes the image of $f$ under $\theta_\lambda \mapsto \theta_{-w_0\lambda}$.

If $f \in \mathcal{A}^W$, then $f^{\sharp} = w_0(\bar{f}) = \bar{f}$, which is central in $\mathcal{H}$. Thus $\check{f} = T_{w_0}^{-1}\,\bar{f}\,T_{w_0} = \bar{f}$.
\end{proof}

Since $\check{e}_K = e_K$ and $\bar{f}$ commutes with $e_K$, Proposition \ref{prop:anti_involution_center} immediately implies
\begin{equation}\label{eq:anti_involution_spherical}
(f e_K)\check{\vphantom{f}} = \bar{f}e_K \qquad (f \in \mathcal{A}^W).
\end{equation}

For a smooth representation $V$, let $\widetilde{V}$ denote the smooth dual.

\begin{lemma}[Contragredient parameter]\label{lem:contragredient}
For $f \in \mathcal{A}^W$, write $\bar{f}$ for its image under $\theta_\lambda \mapsto \theta_{-\lambda}$ and set $\bar{\chi}(f) := \chi(\bar{f})$. Then $\widetilde{V_G(\chi)} \cong V_G(\bar{\chi})$.
\end{lemma}
\begin{proof}
Set $V = V_G(\chi)$. Since $V$ is irreducible and admissible, $\widetilde{V}$ is irreducible and $\dim \widetilde{V}^K = \dim V^K = 1$. For $h \in \mathcal{H}_K$, $\lambda \in \widetilde{V}^K$, and $v \in V^K$, the contragredient action satisfies $(h\cdot\lambda)(v) = \lambda(\check{h}\cdot v)$. For $F \in \mathcal{A}^W$, equation \eqref{eq:anti_involution_spherical} shows that $(Fe_K)\check{\vphantom{f}} = \bar{F}e_K$, so $\mathcal{H}_K$ acts on $\widetilde{V}^K$ through $\bar{\chi}$. By Fact \ref{fact:borel} and Lemma \ref{lem:spherical_unique}, $\widetilde{V} \cong V_G(\bar{\chi})$.
\end{proof}

\paragraph{Assumption.} Throughout Sections \ref{sec:bbf}--\ref{sec:cs}, we assume that the half-sum of positive absolute coroots satisfies $\rho^\vee \in X = X_*(S)$. Since $\rho^\vee$ is $\sigma$-invariant, this is equivalent to $\rho^\vee \in X_*(T)$. The element $t_{\rho^\vee} \in T(F)$ and the weights $\theta_{\pm\rho^\vee} \in \mathcal{A}$ then exist. We remove this assumption in Section \ref{sec:extended} by passing to an extended central quotient group $G''$.

\section[The Sign Idempotent and a New Proof of the BBF Identity]{\fontsize{15}{18}\selectfont The Sign Idempotent and a New Proof of the BBF Identity}\label{sec:bbf}
This section carries out the second computation announced in the introduction: the explicit description of $e_K\mathcal{H}e_{\mathrm{sgn}}$ inside the free module $\mathcal{H}e_{\mathrm{sgn}} = \mathcal{A}e_{\mathrm{sgn}}$ identified in Lemma \ref{lem:HIK}.

We define operators $\mathcal{L}_{w}: \mathcal{A} \to \mathcal{A}$, $w \in W$, through the free module $\mathcal{H}e_{\mathrm{sgn}} = \mathcal{A}e_{\mathrm{sgn}}$ identified in Lemma \ref{lem:HIK}: for $F \in \mathcal{A}$, there is a unique $\mathcal{L}_{w}(F) \in \mathcal{A}$ with
\begin{equation}\label{eq:Lw_def}
T_{w}\,Fe_{\mathrm{sgn}} = \mathcal{L}_{w}(F)\,e_{\mathrm{sgn}}.
\end{equation}
Each $\mathcal{L}_w$ is $\mathbb{C}$-linear, and $\mathcal{L}_{uv} = \mathcal{L}_{u}\circ\mathcal{L}_{v}$ whenever $\ell(uv) = \ell(u) + \ell(v)$ because then $T_{uv} = T_{u}T_{v}$. For a simple reflection $s_\alpha$, $\alpha \in \Delta$, with associated simple root $a = \tilde\alpha \in \Sigma$, the cross-relation \eqref{eq:cross_relation} and the identity $T_{s_\alpha}e_{\mathrm{sgn}} = -e_{\mathrm{sgn}}$ in Lemma \ref{lem:idempotents} make \eqref{eq:Lw_def} explicit:
\begin{align}\label{eq:Bernstein_esgn}
T_{s_\alpha} Fe_{\mathrm{sgn}} &= \Big(s_{\alpha}(F)\, T_{s_{\alpha}} + \mathcal{G}_a\cdot\big(F - s_\alpha(F)\big)\Big)e_{\mathrm{sgn}} \notag \\
&= \Big(-s_\alpha(F) + \mathcal{G}_a\cdot\big(F - s_\alpha(F)\big)\Big)e_{\mathrm{sgn}},
\end{align}
so that $\mathcal{L}_{s_{\alpha}}(F) = -s_\alpha(F) + \mathcal{G}_a\cdot\big(F - s_\alpha(F)\big)$. 

\begin{lemma}\label{lem:H_e_sgn}
An element $F e_{\mathrm{sgn}}\in \mathcal{A}e_{\mathrm{sgn}}$ belongs to $e_K \mathcal{H} e_{\mathrm{sgn}}$ if and only if for every $\alpha \in \Delta$, with $a = \tilde\alpha$ and $c_i = c_i(a)$,
$$s_{\alpha}(F) = \frac{\big(\theta_{-a^\vee} - c_1\big)\big(\theta_{-a^\vee} + c_2\big)}{\big(1 - c_1\theta_{-a^\vee}\big)\big(1 + c_2\theta_{-a^\vee}\big)}\, F.$$
When $a$ is not divisible, so that $(c_1, c_2) = (q_a, 1)$, the ratio reduces to $\dfrac{\theta_{-a^\vee} - q_a}{1 - q_a\theta_{-a^\vee}}$.
\end{lemma}
\begin{proof}
The element $F e_{\mathrm{sgn}}$ lies in $e_K \mathcal{H} e_{\mathrm{sgn}}$ if and only if $e_{K}F e_{\mathrm{sgn}} = F e_{\mathrm{sgn}}$. By Lemma \ref{lem:idempotents}, this is equivalent to $T_{s_{\alpha}} F e_{\mathrm{sgn}} = q_a F e_{\mathrm{sgn}}$ for all simple roots $\alpha$.

Combining $T_{s_{\alpha}} F e_{\mathrm{sgn}} = q_a F e_{\mathrm{sgn}}$ with \eqref{eq:Bernstein_esgn}, and comparing $\mathcal{A}$-coordinates through the injectivity of $P \mapsto Pe_{\mathrm{sgn}}$ (Lemma \ref{lem:HIK}), we obtain
\begin{align*}
-s_\alpha(F) + \mathcal{G}_a\cdot\big(F - s_\alpha(F)\big) = q_a F, \qquad \text{that is,} \qquad s_\alpha(F)\,\big(1 + \mathcal{G}_a\big) = F\,\big(\mathcal{G}_a - q_a\big).
\end{align*}
Write $\theta := \theta_{a^\vee}$. After clearing the denominator $1 - \theta_{2a^\vee}$ of $\mathcal{G}_a$, the two brackets factor:
$$\big(\mathcal{G}_a - q_a\big)\big(1-\theta_{2a^\vee}\big) = q_a\theta^2 + (c_1-c_2)\theta - 1 = \big(c_1\theta - 1\big)\big(c_2\theta + 1\big),$$
$$\big(1 + \mathcal{G}_a\big)\big(1-\theta_{2a^\vee}\big) = q_a + (c_1-c_2)\theta - \theta^2 = -\big(\theta - c_1\big)\big(\theta + c_2\big).$$
Both factorizations use $c_1c_2 = q_a$. Hence the condition reads
$$s_\alpha(F) = \frac{\big(c_1\theta-1\big)\big(c_2\theta+1\big)}{-\big(\theta-c_1\big)\big(\theta+c_2\big)}\,F = \frac{\big(\theta^{-1} - c_1\big)\big(\theta^{-1} + c_2\big)}{\big(1 - c_1\theta^{-1}\big)\big(1 + c_2\theta^{-1}\big)}\,F,$$
where the second equality multiplies numerator and denominator by $\theta^{-2}$. Every step is reversible, so the displayed relation for all simple $\alpha$ is equivalent to $F e_{\mathrm{sgn}} \in e_K\mathcal{H}e_{\mathrm{sgn}}$. For the last claim, cancel the common factor $(\theta^{-1}+1)/(1+\theta^{-1}) = 1$.
\end{proof}

The next lemma controls denominators in $\mathcal{A}$. It is used in the proof of Proposition \ref{prop:eKHe_sgn}.

\begin{lemma}[Coprimality]\label{lem:coprime}
Let $L$ be a field and $\mathcal{A}_L := L[X]$ the group algebra of $X$ over $L$, a unique factorization domain. Let $\mu, \nu \in X$ be non-zero, and let $c, c' \in L^\times$.
\begin{enumerate}
\item If $1 - c\,\theta_\mu$ and $1 - c'\,\theta_\nu$ have a common non-unit divisor in $\mathcal{A}_L$, then $\mu$ and $\nu$ are proportional.
\item If $\nu = \mu$, they have a common non-unit divisor only if $c' = c$. If $\nu = -\mu$, they have a common non-unit divisor only if $c' = c^{-1}$.
\end{enumerate}
\end{lemma}
\begin{proof}
For a non-zero $f \in \mathcal{A}_L$, let $\operatorname{Newt}(f) \subseteq X \otimes_{\mathbb{Z}} \mathbb{R}$ denote the convex hull of its exponents. Ostrowski's identity \cite[Theorem VI]{Ost75}
$$\operatorname{Newt}(fg) = \operatorname{Newt}(f) + \operatorname{Newt}(g)$$
shows that the Newton polytope of a divisor of $1 - c\,\theta_\mu$ is a Minkowski summand of $[0,\mu]$, hence is either a point or a segment parallel to $\mu$. In the first case the divisor is a non-zero scalar multiple of a monomial, hence a unit. A common non-unit divisor of the two binomials therefore forces $\mu$ and $\nu$ to be proportional. This proves (1).

If $\nu = \mu$ and $c' \neq c$, a common divisor of the two binomials divides their difference $(c' - c)\,\theta_\mu$, which is a unit. If $\nu = -\mu$, then
$$1 - c'\,\theta_{-\mu} = -c'\,\theta_{-\mu}\big(1 - c'^{-1}\theta_\mu\big),$$
so the preceding case shows that $c = c'^{-1}$. This proves (2).
\end{proof}

For $a\in\Sigma$, set
$$f_a := \begin{cases} 1 - q_a\theta_{-a^\vee} & \text{if } a \text{ is not divisible},\\[2pt] \big(1 - c_1(a)\theta_{-a^\vee}\big)\big(1 + c_2(a)\theta_{-a^\vee}\big) & \text{if } a \text{ is divisible}\end{cases}.$$
Here the parameters $q_a, c_i(a)$ are understood through their $W$-invariant extensions to all of $\Sigma$. The substitution $\theta_\lambda\mapsto\theta_{-\lambda}$ sends $f_a$ to $f_{-a}$.

\begin{proposition} \label{prop:eKHe_sgn}
Set
$$\Omega^{+} := \theta_{\rho^\vee} \prod_{a \in \Sigma^+} f_a.$$
Then
$$e_K \mathcal{H} e_{\mathrm{sgn}} = \mathcal{A}^W\,\Omega^+e_{\mathrm{sgn}},$$
a free $\mathcal{A}^W$-module of rank one.
In particular, for split $G$, this is $\Omega^+ = \theta_{\rho^\vee}\prod_{\alpha\in\Phi^+}(1 - q\theta_{-\alpha^\vee})$.
\end{proposition}
\begin{proof}
Fix $\alpha \in \Delta$ and put $a = \tilde\alpha$, $c_i = c_i(a)$. The reflection $s_\alpha$ permutes $\Sigma^+ \setminus \{a\}$ and preserves the functions $m$, $q$, $c_1$, $c_2$, so it permutes the factors $f_b$ with $b \neq a$. By \eqref{eq:rho_pairings}, $s_\alpha(\theta_{\rho^\vee}) = \theta_{\rho^\vee}\, \theta_{-m_a a^\vee}$. For $m_a = 1$,
$$s_\alpha\big(\theta_{\rho^\vee} f_a\big) = \theta_{\rho^\vee}\theta_{-a^\vee}\big(1 - q_a\theta_{a^\vee}\big) = \theta_{\rho^\vee}\big(\theta_{-a^\vee} - q_a\big),$$
and for $m_a = 2$,
$$s_\alpha\big(\theta_{\rho^\vee} f_a\big) = \theta_{\rho^\vee}\theta_{-2a^\vee}\big(1 - c_1\theta_{a^\vee}\big)\big(1 + c_2\theta_{a^\vee}\big) = \theta_{\rho^\vee}\big(\theta_{-a^\vee} - c_1\big)\big(\theta_{-a^\vee} + c_2\big),$$
so in both cases $s_\alpha(\Omega^+)/\Omega^+$ is exactly the ratio in Lemma \ref{lem:H_e_sgn}. It follows that $\Omega^{+}e_{\mathrm{sgn}} \in e_K \mathcal{H} e_{\mathrm{sgn}}$, and the same lemma shows $P\,\Omega^{+}e_{\mathrm{sgn}} \in e_K \mathcal{H} e_{\mathrm{sgn}}$ for every $P \in \mathcal{A}^W$.

Conversely, let $F\in\mathcal A$ and suppose that $Fe_{\mathrm{sgn}}\in e_K\mathcal He_{\mathrm{sgn}}$. Define $P := F / \Omega^{+} \in \mathrm{Frac}(\mathcal{A})$. For any simple root $\alpha$, put $a=\tilde\alpha$ and $c_i=c_i(a)$. Then $s_{\alpha}(F) = s_{\alpha}(P) s_{\alpha}(\Omega^{+})$. The first equality below follows from Lemma \ref{lem:H_e_sgn}, and the second from the calculation above:
$$
\frac{s_\alpha(F)}{F}
=
\frac{\big(\theta_{-a^\vee} - c_1\big)\big(\theta_{-a^\vee} + c_2\big)}
{\big(1 - c_1\theta_{-a^\vee}\big)\big(1 + c_2\theta_{-a^\vee}\big)}
=
\frac{s_\alpha(\Omega^+)}{\Omega^+}.
$$
Since $F=P\Omega^+$, cancellation shows that $s_\alpha(P)=P$. Hence $P$ is $W$-invariant.

It remains to show that $P$ lies in $\mathcal{A}$ and not merely in $\mathrm{Frac}(\mathcal{A})$. The factors of $\Omega^+$ other than the unit $\theta_{\rho^\vee}$ are among the elements $1 - c\,\theta_{-a^\vee}$ with $a \in \Sigma^+$ and $c \in \{q_a\}$ or $\{c_1(a), -c_2(a)\}$. We call these constants \emph{admissible}, and each is, up to sign, a positive integral power of $q$.

Now write $P = P_1/P_2$ in lowest terms, with $P_1, P_2 \in \mathcal{A}$ coprime. From $P_1\,\Omega^{+} = P_2 F$, we get $P_2 \mid \Omega^{+}$. Suppose $P_2$ is not a unit and let $d$ be an irreducible factor of $P_2$. Since the unit $\theta_{\rho^\vee}$ contributes nothing, $d$ divides $1 - c\,\theta_{-a^\vee}$ for some $a \in \Sigma^+$ and some admissible $c$.

Since $s_a(P)=P$, the fractions $P_1/P_2$ and $s_a(P_1)/s_a(P_2)$ are reduced representations of the same element. Thus $s_a(P_2)$ is an associate of $P_2$. It follows that the irreducible element $s_a(d)$ divides $P_2$ and hence some factor $1-c'\theta_{-b^\vee}$ of $\Omega^+$, where $b \in \Sigma^+$ and $c'$ is admissible. It also divides
$$s_a(1-c\theta_{-a^\vee})=1-c\theta_{a^\vee}.$$
Lemma \ref{lem:coprime}(1) shows that $a^\vee$ and $b^\vee$ are proportional. Since the coroot system is reduced and both are positive, they are equal. Lemma \ref{lem:coprime}(2), applied with $\nu=-\mu=-a^\vee$, then forces $c'=c^{-1}$. This is impossible because every admissible constant has the form $\pm q^n$ with $n\geq1$, whereas its inverse has negative $q$-exponent. Hence $P_2$ is a unit and $P \in \mathcal{A}^W$.

Freeness holds because $\mathcal{A}$ is a domain and $F \mapsto Fe_{\mathrm{sgn}}$ is injective on $\mathcal{A}$ (Lemma \ref{lem:HIK}).
\end{proof}

The main theorem of this section generalizes an identity of Brubaker, Bump, and Friedberg, which we first recall. Suppose $G$ is split, with all parameters equal to $q$, so that every $m_a = 1$ and every $f_a = 1 - q\,\theta_{-a^\vee}$. They proved an evaluation for each linear character of the finite Hecke algebra $\mathcal{H}_0$ \cite[Theorem 5]{BBF16}. In the coordinates of this section, the sign-character case reads
$$\sum_{w \in W} \mathcal{L}_{w}(\theta_\lambda) \,=\, \left(\prod_{a \in \Sigma^+} \frac{1 - q\,\theta_{-a^\vee}}{1 - \theta_{-a^\vee}}\right) \mathrm{Alt}(\theta_\lambda), \qquad \lambda \in X,$$
the split equal-parameter case of the theorem below, one monomial $F = \theta_\lambda$ at a time. Brubaker, Bump, and Friedberg express the induced-sign Hecke action in terms of Demazure operators and specialize at $q=0$, where the spherical symmetrizer becomes the longest Demazure operator. They then use the equality of the Demazure and Weyl character formulas. In the proof below, Proposition \ref{prop:eKHe_sgn} replaces BBF's Demazure-operator analysis of the spherical symmetrizer with the rank-one description of $e_K\mathcal H e_{\mathrm{sgn}}$. A degree argument shows that the resulting $W$-invariant factor is independent of $q$, and specialization at $q=1$ determines it through the Weyl denominator formula in Fact \ref{fact:denominator}.

\begin{theorem}[Algebraic BBF identity] \label{thm:bbf_generalized}
Given $F \in \mathcal{A}$,
\begin{equation*}
\sum_{w \in W} \mathcal{L}_{w}(F) = P(F) \cdot \Omega^{+} = \left( \prod_{a \in \Sigma^+} \frac{f_a}{1 - \theta_{-m_a a^\vee}} \right) \mathrm{Alt}(F),
\end{equation*}
where $\displaystyle P(F) = \frac{\mathrm{Alt}(F)}{ \mathrm{Alt}(\theta_{\rho^{\vee}})} \in \mathcal{A}^{W}$. Consequently,
$$e_{K}Fe_{\mathrm{sgn}} = \frac{1}{W(q)}\cdot\frac{\mathrm{Alt}(F)}{\mathrm{Alt}(\theta_{\rho^{\vee}})}\Omega^{+} e_{\mathrm{sgn}}.$$
\end{theorem}
\begin{proof}
Fix $F\in\mathbb C[X]$ and regard $q$ as an indeterminate. Recall the $W$-invariant function $d:\Sigma\to\mathbb Z_{>0}$ defined by $q_a=q^{d(a)}$, with $d(\tilde\alpha)=d_\alpha$ for every $\alpha\in\Delta$. Replace $\mathcal H$ by the generic affine Hecke algebra over $\mathbb C(q)$ obtained from the Bernstein presentation by reading $q_a=q^{d(a)}$ and $q_{a*}$ as the corresponding power prescribed in Section \ref{subsec:bernstein}. The constants $c_1(a),c_2(a)$ are then integral powers of $q$, and the numerators of $\mathcal G_a$ lie in $\mathbb C[q][X]$. We retain the notation $\mathcal H$ for this generic algebra and put $\mathcal A=\mathbb C(q)[X]$.

After the faithful scalar extension $\mathbb{C}(q)\subseteq\mathbb{C}(v)$ with $q=v^2$, the Bernstein basis theorem \cite[Proposition 3.7]{Lus89} shows that $\{\theta_\lambda T_w\}$ is a $\mathbb C(v)$-basis. It is therefore already a $\mathbb C(q)$-basis. The proof of Lemma \ref{lem:HIK} applies, and the idempotents $e_K,e_{\mathrm{sgn}}$ are defined because $W(q)$ and $W(q^{-1})$ are invertible in $\mathbb C(q)$. Thus \eqref{eq:Lw_def} defines the operators $\mathcal L_w$ on $\mathbb C(q)[X]$ with the same formula \eqref{eq:Bernstein_esgn}. Lemma \ref{lem:H_e_sgn} and Proposition \ref{prop:eKHe_sgn} also apply over $\mathbb C(q)$. Indeed, Lemma \ref{lem:coprime} applies over this field, and the admissible constants are, up to sign, positive powers of $q$, so none is the inverse of another.

Let $L(F) = \sum_{w \in W} \mathcal{L}_w(F)$. Then $L(F)e_{\mathrm{sgn}} = W(q)e_{K}Fe_{\mathrm{sgn}} \in e_{K}\mathcal{H}e_{\mathrm{sgn}}$. By the generic version of Proposition \ref{prop:eKHe_sgn}, we can write
$$L(F) = P\Omega^{+}$$
for some $P=P(F) \in \mathbb{C}(q)[X]^{W}$.

Both $L(F)$ and $\Omega^+$ lie in $\mathbb C[q][X]$. The coefficient of $\theta_{\rho^\vee}$ in $\Omega^+$ is $1$, so $\Omega^+$ is primitive over $\mathbb C[q]$. Since it divides $L(F)$ in $\mathbb C(q)[X]$, Gauss's lemma for Laurent polynomial rings shows that
\[P\in\mathbb C[q][X]^W.\]

Put $N:=\sum_{a\in\Sigma^+}d(a)$. The parameter formulas of Section \ref{subsec:bernstein} show that $\deg_q f_a=d(a)$ and hence $\deg_q\Omega^+=N$. They also imply
$$\deg_q\mathcal L_{s_\alpha}(G)\leq d(a)+\deg_qG \qquad \bigl(G\in\mathbb C[q][X],\ a=\tilde\alpha\bigr),$$
because the two numerator constants of $\mathcal G_a$ have degree at most $d(a)$. If $w=s_{\alpha_1}\cdots s_{\alpha_r}$ is reduced and $a_i=\tilde\alpha_i$, the $W$-invariance of $d$ allows the estimate
$$\deg_q\mathcal L_w(F)\leq\sum_{i=1}^r d(a_i)
=\sum_{a\in\Sigma^+\cap w^{-1}\Sigma^-}d(a)\leq N.$$
Therefore $\deg_qL(F)\leq N$. Since $L(F)=P\Omega^+$ and $\deg_q\Omega^+=N$, the polynomial $P$ is independent of $q$.

Finally, we determine $P$ by specializing $q = 1$. Every parameter and every $c_i(a)$ then becomes $1$, so both numerator constants of $\mathcal{G}_a$ vanish, $\mathcal{L}_{s_{\alpha}} (F)\vert_{q = 1} = -s_{\alpha}(F)$, and $\mathcal{L}_{w} (F)\vert_{q = 1} = (-1)^{\ell(w)} w(F)$. Hence
$$L(F)\vert_{q = 1} = \mathrm{Alt}(F).$$
The factor $f_a$ of $\Omega^{+}$ becomes $1 - \theta_{-a^\vee}$ for $m_a = 1$ and $(1-\theta_{-a^\vee})(1+\theta_{-a^\vee}) = 1 - \theta_{-2a^\vee}$ for $m_a = 2$, so by Fact \ref{fact:denominator}
$$\Omega^{+}\vert_{q=1} = \theta_{\rho^\vee} \prod_{a \in \Sigma^+} \big(1 - \theta_{-m_a a^\vee}\big) = \mathrm{Alt}(\theta_{\rho^{\vee}}).$$
Therefore,
\[P = \frac{L(F)\vert_{q = 1}}{\Omega^{+}\vert_{q=1}} = \frac{\mathrm{Alt}(F)}{ \mathrm{Alt}(\theta_{\rho^{\vee}})}.\]
Substituting the definitions of $\Omega^+$ and $\mathrm{Alt}(\theta_{\rho^\vee})$ produces the first displayed identity of the theorem. Specializing the indeterminate $q$ to the residue cardinality yields this identity in the original algebra $\mathcal A$. Finally,
$$L(F)e_{\mathrm{sgn}}=\Bigl(\sum_{w\in W}T_w\Bigr)Fe_{\mathrm{sgn}}=W(q)e_KFe_{\mathrm{sgn}},$$
and division by $W(q)$ proves the second displayed identity.
\end{proof}

\section[\texorpdfstring{$K$}{K}-Invariants of the Gelfand--Graev Representation, Conductor \texorpdfstring{$\mathfrak{p}$}{p}]{\fontsize{15}{18}\selectfont \texorpdfstring{$K$}{K}-Invariants of the Gelfand--Graev Representation, Conductor \texorpdfstring{$\mathfrak{p}$}{p}}\label{sec:ggp}

Under the assumption $\rho^\vee \in X$, we have $X^{++} = \rho^\vee + X^{+}$. In this section, we show that $(\mathrm{ind}_{\overline{U}}^{G}\overline{\psi})^{K}$ is a free right $\mathcal{H}_{K}$-module of rank one, by projecting the free module $(\mathrm{ind}_{\overline{U}}^{G}\overline{\psi})^{I} \cong \mathcal{A}e_{\mathrm{sgn}}$ identified in Fact \ref{fact:CS}(2) onto $K$-invariants and computing the projection through the description of $e_K\mathcal{H}e_{\mathrm{sgn}}$ in Section \ref{sec:bbf}.

For $x\in X^{++}$, write
$$\phi_x:=\Phi^{\overline\psi}_x,$$
with the notation of Lemma \ref{lem:whittaker_support}(2). Thus $\phi_x$ is supported on $\overline U t_xK$ and satisfies $\phi_x(ut_xk)=\delta_{\overline B}(t_x)^{1/2}\overline\psi(u)$. By the same part, the functions $\{\phi_x\mid x\in X^{++}\}$ form a basis of $(\mathrm{ind}_{\overline U}^G\overline\psi)^K$.

\begin{lemma}\label{lem:Kt_x_I_TUbar}
Let $x \in X^{--}$ and $y \in X$. If $u \in \overline{U}$ satisfies $t_{y} u \in K t_{x} I$, then $x = y$ and $u \in \overline{U}(\mathfrak p)$.
\end{lemma}

\begin{proof}
We follow the geometric implication in the proof of \cite[Lemma 5.3]{GK22}, translated to our opposite-Borel conventions. Since $x$ is anti-dominant, $t_{x} U(\mathcal O) t_{x}^{-1} \subseteq U(\mathcal O) \subseteq K$, so the Iwahori factorization $I = U(\mathcal O)T(\mathcal O)\overline{U}(\mathfrak p)$ reduces the double coset to $K t_{x} I = K t_{x} \overline{U}(\mathfrak p)$. Write $t_{y}u = kt_{x}u'$ with $k \in K$ and $u' \in \overline{U}(\mathfrak p)$. Then
$$k^{-1} = t_{x}u'u^{-1}t_{y}^{-1} = \big(t_{x}t_{y}^{-1}\big)\big(t_{y}u'u^{-1}t_{y}^{-1}\big) \in T(\mathcal O)\overline{U}(\mathcal O).$$
The first factor forces $x = y$. The second, together with the strict anti-dominance of $x$, implies $u'u^{-1} \in t_{x}^{-1}\overline{U}(\mathcal O)t_{x} \subseteq \overline{U}(\mathfrak p)$, so $u \in \overline{U}(\mathfrak p)$.
\end{proof}

The next proposition is essentially the twisted Satake transform in \cite{Gu13, GK22}.

\begin{proposition}\label{prop:twisted_satake}
For $x \in X^{++}$,
    $$e_{K}\cdot(\theta_{-x}\cdot \mathrm{ch}_{I}^{\overline{\psi}}) = \frac{q_{w_0}}{W(q)}\,\phi_{x}. $$
\end{proposition}

\begin{proof}
Here $\theta_{-x} = q_{t_{-x}}^{-1/2}\,\mathbf{1}_{It_{-x}I} = \delta_{\overline{B}}(t_{x})^{1/2}\,\mathbf{1}_{It_{-x}I}$ and $e_{K} = \frac{1}{W(q)}\mathbf{1}_{K}$, so
\begin{equation}\label{eq:eK_theta_x}
e_{K} * \theta_{-x} = \frac{1}{W(q)}\,\delta_{\overline{B}}(t_{x})^{1/2}\mathbf{1}_K * \mathbf{1}_{It_{-x}I} = \frac{q_{w_0}}{W(q)}\,\delta_{\overline{B}}(t_{x})^{1/2} \mathbf{1}_{Kt_{-x}I}.
\end{equation}
Indeed, $\mathbf{1}_K * \mathbf{1}_{It_{-x}I}$ equals $\mathrm{vol}\big(K \cap t_{-x}\, I t_{x} I\big)\,\mathbf{1}_{Kt_{-x}I}$. The Iwahori factorization and the dominance of $x$ imply $I t_x I = \overline{U}(\mathfrak{p})\, T(\mathcal{O})\, t_x\, I$, hence
$$t_{-x}\, I t_x I = \big(t_{-x}\overline{U}(\mathfrak{p})\, t_{x}\big)\, T(\mathcal{O})\, I.$$
Using the strict dominance of $x$, direct spanning, and conjugation on the affine root groups, one sees that $t_{-x}\overline U(\mathfrak p)t_x\supseteq\overline U(\mathcal O)$ and that its intersection with $K$ is exactly $\overline U(\mathcal O)$. Consequently, $K \cap t_{-x}It_xI = \overline{U}(\mathcal{O})\,I$, of volume $[\overline{U}(\mathcal{O}) : \overline{U}(\mathfrak{p})] = q_{w_0}$.

Both sides of the proposition lie in $(\mathrm{ind}_{\overline{U}}^{G}\overline{\psi})^{K}$, so they are determined by their values at $t_y$ for $y\in X$. By \eqref{eq:eK_theta_x},
$$e_{K}\cdot(\theta_{-x}\cdot \mathrm{ch}_{I}^{\overline{\psi}})(t_{y}) = \frac{q_{w_0}\,\delta_{\overline{B}}(t_{x})^{1/2}}{W(q)}\int_{Kt_{-x}I}\mathrm{ch}_{I}^{\overline{\psi}}(t_{y}h)\,dh.$$
The integrand vanishes unless $h = t_{-y}\,u\,i$ with $u \in \overline{U}$ and $i \in I$. Then $t_{-y}u \in Kt_{-x}I$, and Lemma \ref{lem:Kt_x_I_TUbar}, applied to $-x \in X^{--}$, forces $y = x$ and $u \in \overline{U}(\mathfrak p)$. For $y = x$, the integrand equals $1$ exactly on $t_{-x}I \subseteq Kt_{-x}I$, so the integral is $\mathrm{vol}(t_{-x}I) = 1$ and
\[e_{K}\cdot(\theta_{-x}\cdot \mathrm{ch}_{I}^{\overline{\psi}})(t_{x}) = \frac{q_{w_0}}{W(q)}\,\delta_{\overline{B}}(t_{x})^{1/2} = \frac{q_{w_0}}{W(q)}\,\phi_{x}(t_{x}). \qedhere\]
\end{proof}

For $x \in X^+$, write
\[ h_x := (\mathrm{ch}V_x)e_K \in \mathcal{H}_K \]
for the image under the Satake isomorphism of $\mathrm{ch}V_x \in \mathcal{A}^W$ (Fact \ref{fact:chV}(2)).

\begin{theorem}\label{thm:chVx_action_right}
For $x \in X^+$,
\[ \phi_{\rho^\vee} * h_x = \phi_{x + \rho^\vee}. \]
\end{theorem}

\begin{proof}
Set $F = \mathrm{ch}V_x \in \mathcal{A}^W$, so that $h_x=Fe_K$. For $y \in X^{++}$, Proposition \ref{prop:twisted_satake}, the identity $e_{\mathrm{sgn}} \cdot \mathrm{ch}_I^{\overline{\psi}} = \mathrm{ch}_I^{\overline{\psi}}$ in Fact \ref{fact:CS}(1), and Theorem \ref{thm:bbf_generalized} applied to $\theta_{-y}$ combine to produce
\begin{equation}\label{eq:phi_via_bbf}
\begin{aligned}
\phi_y
&= \frac{W(q)}{q_{w_0}}\,\big(e_K\,\theta_{-y}\,e_{\mathrm{sgn}}\big)\cdot\mathrm{ch}_I^{\overline{\psi}}\\
&= \frac{1}{q_{w_0}}\,
\frac{\mathrm{Alt}(\theta_{-y})}{\mathrm{Alt}(\theta_{\rho^\vee})}\,
\Omega^+e_{\mathrm{sgn}}\cdot\mathrm{ch}_I^{\overline{\psi}}.
\end{aligned}
\end{equation}
Applying $\theta_\lambda \mapsto \theta_{-\lambda}$ to the identity
$F\,\mathrm{Alt}(\theta_{\rho^\vee})=\mathrm{Alt}(\theta_{x+\rho^\vee})$ in Fact \ref{fact:chV}(1), we obtain
\[
\bar F\,\mathrm{Alt}(\theta_{-\rho^\vee})
=\mathrm{Alt}(\theta_{-(x+\rho^\vee)}).
\]
Both $\rho^\vee$ and $x+\rho^\vee$ lie in $X^{++}$. The function $\phi_{\rho^\vee}$ is $K$-invariant, so $e_K\cdot\phi_{\rho^\vee}=\phi_{\rho^\vee}$. Hence the right-convolution convention, \eqref{eq:anti_involution_spherical}, and \eqref{eq:phi_via_bbf}, applied at $y=\rho^\vee$ and $y=x+\rho^\vee$, show that
\begin{align*}
\phi_{\rho^\vee}*h_x
&=(Fe_K)\check{\vphantom{f}}\cdot\phi_{\rho^\vee}\\
&=(\bar Fe_K)\cdot\phi_{\rho^\vee}\\
&=\bar F\cdot\phi_{\rho^\vee}\\
&=\frac{1}{q_{w_0}}\,
\frac{\bar F\,\mathrm{Alt}(\theta_{-\rho^\vee})}
{\mathrm{Alt}(\theta_{\rho^\vee})}\,
\Omega^+e_{\mathrm{sgn}}\cdot\mathrm{ch}_I^{\overline{\psi}}\\
&=\frac{1}{q_{w_0}}\,
\frac{\mathrm{Alt}(\theta_{-(x+\rho^\vee)})}
{\mathrm{Alt}(\theta_{\rho^\vee})}\,
\Omega^+e_{\mathrm{sgn}}\cdot\mathrm{ch}_I^{\overline{\psi}}\\
&=\phi_{x+\rho^\vee}. \qedhere
\end{align*}
\end{proof}

\begin{corollary}\label{cor:freeness_p}
Under right convolution, $(\mathrm{ind}_{\overline{U}}^{G}\overline{\psi})^{K}$ is a free right $\mathcal{H}_K$-module of rank one generated by $\phi_{\rho^\vee}$.
\end{corollary}
\begin{proof}
By Fact \ref{fact:chV}(2) and Theorem \ref{thm:spherical_iso}, $\{h_x\mid x\in X^+\}$ is a basis of $\mathcal H_K$. By Lemma \ref{lem:whittaker_support}(2), $\{\phi_{x+\rho^\vee}\mid x\in X^+\}=\{\phi_y\mid y\in X^{++}\}$ is a basis of $(\mathrm{ind}_{\overline U}^G\overline\psi)^K$. The statement follows from Theorem \ref{thm:chVx_action_right}.
\end{proof}

\section{Translation to Conductor \texorpdfstring{$\mathcal{O}$}{O}}\label{sec:conductorO}

We keep the assumption $\rho^\vee \in X$ and transfer the results of Section \ref{sec:ggp} to the character $\overline{\omega}$ of conductor $\mathcal{O}$. Those results hold for any character of conductor $\mathfrak{p}$ in place of $\overline{\psi}$, so we now choose $\psi$ compatibly with $\omega$.

\begin{lemma}\label{lem:conductor_shift}
Define $\psi(u) := \omega(t_{-\rho^\vee}\, u\, t_{\rho^\vee})$ for $u \in \overline{U}$. Then $\psi$ is a character of $\overline{U}$ of conductor $\mathfrak{p}$.
\end{lemma}
\begin{proof}
For $\alpha\in\Delta$, \eqref{eq:root_coordinate_shift} shows that conjugation by $t_{-\rho^\vee}$ acts on the $x_{[-\alpha]}$-coordinate of $U_{-\alpha}/U_{-2\alpha}$ as multiplication by $\varpi^{-\langle\alpha,\rho^\vee\rangle} = \varpi^{-1}$, so
$$\psi_\alpha(c) = \omega_\alpha\big(\varpi^{-1}c\big), \qquad c \in E_\alpha.$$
Since $\omega_\alpha$ has conductor $\mathcal{O}_{E_\alpha}$, the character $\psi_\alpha$ has conductor $\mathfrak{p}_{E_\alpha}$.
\end{proof}

For $f \in \mathrm{ind}_{\overline{U}}^{G}\,\overline{\psi}$, $u \in \overline{U}$, and $g \in G$,
\begin{equation*}
f\big(t_{\rho^\vee}\, u\, g\big) = \overline{\psi}\big(t_{\rho^\vee}\, u\, t_{-\rho^\vee}\big)\, f\big(t_{\rho^\vee}\, g\big) = \overline{\omega}(u)\, f\big(t_{\rho^\vee}\, g\big).
\end{equation*}
Left translation by $t_{\rho^\vee}$ therefore gives the following proposition.

\begin{proposition}\label{prop:translation_iso}
Define $\mathcal{L} : \mathrm{ind}_{\overline{U}}^{G}\, \overline{\psi} \to \mathrm{ind}_{\overline{U}}^{G}\, \overline{\omega}$ by $(\mathcal{L}f)(g) := f(t_{\rho^\vee}\, g)$. Then $\mathcal{L}$ is an isomorphism onto $\mathrm{ind}_{\overline{U}}^{G}\, \overline{\omega}$ commuting with right translation by $G$. In particular, $\mathcal{L}$ restricts to an isomorphism of $K$-invariants which intertwines both the left $\mathcal{H}_K$-action and the right convolution action.
\end{proposition}

\begin{corollary}\label{cor:freeness_O}
For $y \in X^+$, let $\Phi^{\overline{\omega}}_y$ be the function defined in Lemma \ref{lem:whittaker_support}(2). Then
$$\Phi^{\overline{\omega}}_0 * h_y = \Phi^{\overline{\omega}}_y,$$
and $(\mathrm{ind}_{\overline{U}}^G \overline{\omega})^K$ is a free right $\mathcal{H}_K$-module of rank one generated by $\Phi^{\overline{\omega}}_0$.
\end{corollary}
\begin{proof}
With the choice of $\psi$ in Lemma \ref{lem:conductor_shift}, retain the notation $\phi_x=\Phi^{\overline\psi}_x$ of Section \ref{sec:ggp}. For $y\in X^+$, the function $\mathcal{L}(\phi_{y+\rho^\vee})$ is supported on
\[
t_{-\rho^\vee}\,\overline{U}\,t_{y+\rho^\vee}K
=\overline{U}\,t_yK.
\]
For $u\in\overline{U}$ and $k\in K$,
\begin{align*}
\mathcal{L}\big(\phi_{y+\rho^\vee}\big)(ut_yk)
&=\phi_{y+\rho^\vee}
\big((t_{\rho^\vee}ut_{-\rho^\vee})\,t_{y+\rho^\vee}k\big)\\
&= \delta_{\overline{B}}(t_{y+\rho^\vee})^{1/2}\, \overline{\omega}(u)\\
&= \delta_{\overline{B}}(t_{\rho^\vee})^{1/2}\,
   \delta_{\overline{B}}(t_y)^{1/2}\,\overline{\omega}(u).
\end{align*}
Thus, for every $y\in X^+$,
\begin{equation}\label{eq:translated_basis}
\Phi^{\overline\omega}_y
=\delta_{\overline B}(t_{\rho^\vee})^{-1/2}
\mathcal L\big(\phi_{y+\rho^\vee}\big).
\end{equation}
Since $\mathcal{L}$ commutes with right convolution, \eqref{eq:translated_basis} and Theorem \ref{thm:chVx_action_right} imply
\begin{align*}
\Phi^{\overline\omega}_0*h_y
&=\delta_{\overline B}(t_{\rho^\vee})^{-1/2}
  \mathcal L(\phi_{\rho^\vee})*h_y\\
&=\delta_{\overline B}(t_{\rho^\vee})^{-1/2}
  \mathcal L\big(\phi_{\rho^\vee}*h_y\big)\\
&=\delta_{\overline B}(t_{\rho^\vee})^{-1/2}
  \mathcal L\big(\phi_{y+\rho^\vee}\big)\\
&=\Phi^{\overline\omega}_y.
\end{align*}
Freeness follows because $\{h_y\mid y\in X^+\}$ is a basis of $\mathcal H_K$ by Fact \ref{fact:chV}(2) and Theorem \ref{thm:spherical_iso}, and $\{\Phi^{\overline\omega}_y\mid y\in X^+\}$ is a basis of $(\mathrm{ind}_{\overline U}^G\overline\omega)^K$ by Lemma \ref{lem:whittaker_support}(2).
\end{proof}

\section[The Casselman--Shalika Formula and the Genericity Criterion]{\fontsize{15}{18}\selectfont The Casselman--Shalika Formula and the Genericity Criterion}\label{sec:cs}

We continue to assume $\rho^\vee \in X$. Pairing spherical Whittaker functions against the free modules of Sections \ref{sec:ggp} and \ref{sec:conductorO}, with $h_y = (\mathrm{ch}V_y)e_K$, yields the Casselman--Shalika formula. The space $e_{\mathrm{sgn}}\mathcal{H}e_K$ then determines which spherical representations admit Whittaker models.

\subsection{The induction--compact induction pairing}

This subsection does not use the assumption $\rho^\vee \in X$. Let $\tau$ be a character of $\overline{U}$. For $f_1 \in \mathrm{Ind}_{\overline{U}}^G \tau$ and $f_2 \in \mathrm{ind}_{\overline{U}}^G \overline{\tau}$, the product $f_1 f_2$ lies in $C_c(\overline{U}\backslash G)$, and we set
\begin{equation}\label{eq:pairing_def}
\langle f_1, f_2 \rangle := \int_{\overline{U}\backslash G} f_1(g) f_2(g)\, dg = \sum_{y \in X} \delta_{\overline{B}}(t_y)^{-1} \int_K f_1(t_y k)\, f_2(t_y k)\, dk,
\end{equation}
where $dk$ is the restriction to $K$ of the chosen Haar measure on $G$, so that $\int_K dk = W(q)$. The second equality follows from the Iwasawa integration formula after integrating out $\overline U$, and the resulting functional is invariant under right translation \cite[Section IV.1]{Ca}.

\begin{lemma}\label{lem:pairing}
Let $f_1 \in (\mathrm{Ind}_{\overline{U}}^G \tau)^K$.
\begin{enumerate}
\item For every $f_2 \in \mathrm{ind}_{\overline{U}}^G \overline{\tau}$ and $H \in \mathcal{H}_K$, define the right convolution by
$$ (f_2 * H)(g) = \int_G f_2(gh^{-1})H(h)\,dh. $$
With the left action of Section \ref{sec:notation},
$$\langle f_1, f_2 * H \rangle = \langle H \cdot f_1, f_2 \rangle.$$
\item If $\tau$ has conductor $\mathcal{O}$ or $\mathfrak{p}$, then for every $y\in X(\tau)$,
$$\big\langle f_1,\, \Phi^{\overline{\tau}}_y \big\rangle = W(q)\,\delta_{\overline{B}}(t_y)^{-1/2}\, f_1(t_y).$$
\end{enumerate}
\end{lemma}
\begin{proof}
For (1), exchange the integrals and, by the right invariance of \eqref{eq:pairing_def}, substitute $g \mapsto gh$. For (2), the function $\Phi^{\overline{\tau}}_y$ is supported on the single coset $\overline{U}t_yK$ with value $\delta_{\overline{B}}(t_y)^{1/2}\overline{\tau}(u)$ at $ut_yk$. Hence only the term $y$ contributes to \eqref{eq:pairing_def}, and since $f_1$ is right $K$-invariant,
\[\langle f_1, \Phi^{\overline{\tau}}_y \rangle = \delta_{\overline{B}}(t_y)^{-1}\int_K f_1(t_yk)\,\Phi^{\overline{\tau}}_y(t_yk)\,dk = \delta_{\overline{B}}(t_y)^{-1}\cdot f_1(t_y)\cdot \delta_{\overline{B}}(t_y)^{1/2}\cdot W(q). \qedhere\]
\end{proof}

\subsection{The Casselman--Shalika formula}

\begin{definition}
Let $\chi : \mathcal{H}_K \to \mathbb{C}$ be an algebra character and $\tau$ a character of $\overline{U}$. A \emph{spherical Whittaker function} with parameter $\chi$ is a non-zero $\mathcal{W} \in (\mathrm{Ind}_{\overline{U}}^G \tau)^K$ satisfying $h \cdot \mathcal{W} = \chi(h)\, \mathcal{W}$ for all $h \in \mathcal{H}_K$.
\end{definition}

Via the Satake isomorphism $\mathcal A^W \xrightarrow{\sim} \mathcal H_K$, $F \mapsto Fe_K$, we also use $\chi$ for its pullback to $\mathcal A^W$. Thus $\chi(F)$ means $\chi(Fe_K)$ for $F \in \mathcal A^W$.

\begin{theorem}[Casselman--Shalika formula, conductor $\mathcal{O}$]\label{thm:CS_O}
Let $0 \neq \mathcal{W} \in (\mathrm{Ind}_{\overline{U}}^G\, \omega)^K$ be a spherical Whittaker function with parameter $\chi$. Then $\mathcal{W}(t_y) = 0$ for $y \notin X^+$, and
\[
\mathcal{W}(t_y) = \delta_{\overline{B}}(t_y)^{1/2}\, \chi(\mathrm{ch}V_y)\, \mathcal{W}(e) \qquad (y \in X^+).
\]
In particular, $\mathcal{W}(e) \neq 0$.
\end{theorem}
\begin{proof}
The vanishing off $X^+$ is Lemma \ref{lem:whittaker_support}(1). Fix $y \in X^+$ and pair $\mathcal{W}$ against the basis element $\Phi^{\overline\omega}_y$ defined in Lemma \ref{lem:whittaker_support}(2). Lemma \ref{lem:pairing}(2) evaluates the first pairing below. Corollary \ref{cor:freeness_O} replaces $\Phi^{\overline\omega}_y$ by $\Phi^{\overline\omega}_0*h_y$, and Lemma \ref{lem:pairing}(1), together with the eigenfunction relation, transfers $h_y$ to $\mathcal W$:
\begin{align*}
W(q)\,\delta_{\overline{B}}(t_y)^{-1/2}\, \mathcal{W}(t_y) = \big\langle \mathcal{W}, \Phi^{\overline{\omega}}_y \big\rangle
= \big\langle \mathcal{W}, \Phi^{\overline{\omega}}_0 * h_y \big\rangle
= \chi(\mathrm{ch}V_y)\, \big\langle \mathcal{W}, \Phi^{\overline{\omega}}_0 \big\rangle.
\end{align*}
Applying Lemma \ref{lem:pairing}(2) with $y = 0$ shows that the last pairing equals $W(q)\, \mathcal{W}(e)$. Dividing by $W(q)\,\delta_{\overline{B}}(t_y)^{-1/2}$ proves the formula.
\end{proof}

By the same proof, using $\phi_{\rho^\vee}$ and Theorem \ref{thm:chVx_action_right} in place of $\Phi^{\overline{\omega}}_0$ and Corollary \ref{cor:freeness_O}, we have the following conductor $\mathfrak p$ version.

\begin{theorem}[Casselman--Shalika formula, conductor $\mathfrak{p}$]\label{thm:CS_p}
Let $0 \neq \mathcal{W} \in (\mathrm{Ind}_{\overline{U}}^G\, \psi)^K$ be a spherical Whittaker function with parameter $\chi$. Then $\mathcal{W}(t_x) = 0$ for $x \notin X^{++}$, and
\[
\mathcal{W}(t_x) = \delta_{\overline{B}}(t_{x-\rho^\vee})^{1/2}\, \chi(\mathrm{ch}V_{x-\rho^\vee})\, \mathcal{W}(t_{\rho^\vee}) \qquad (x \in X^{++}).
\]
\end{theorem}

\subsection{The genericity criterion}

For a character $\tau$ of $\overline{U}$, we say that an irreducible smooth representation $V$ is \emph{$\tau$-generic} if $\mathrm{Hom}_{\overline{U}}(V, \tau) \neq 0$. Equivalently, $V$ is $\tau$-generic if $\mathrm{Hom}_G(V, \mathrm{Ind}_{\overline{U}}^G\tau) \neq 0$. Recall from Lemma \ref{lem:spherical_unique} and Definition \ref{def:spherical_rep} the simple $\mathcal{H}$-module $M_G(\chi)$ and the spherical representation $V_G(\chi)$ attached to an algebra character $\chi$ of $\mathcal{H}_K$. The spherical vector $v_K$ spans $V_G(\chi)^K = e_K M_G(\chi)$. We determine which spherical representations are generic.

Recall the Iwahori--Matsumoto involution $\iota$ of Lemma \ref{lem:IM} and the elements $f_a$, $a\in\Sigma$, defined in Section \ref{sec:bbf}.

\begin{proposition}\label{prop:esgnHeK}
Set
$$\Omega^{-} := \iota(\Omega^{+}) = \theta_{-\rho^\vee}\prod_{a\in\Sigma^+}f_{-a}.$$
Then
$$e_{\mathrm{sgn}}\mathcal{H}e_K = \mathcal{A}^W\,\Omega^{-} e_K.$$
The module $e_{\mathrm{sgn}}\mathcal{H}e_K$ is therefore free over $\mathcal{A}^W$ of rank one. Moreover, $e_{\mathrm{sgn}}\,\Omega^- e_K = \Omega^- e_K$.
\end{proposition}
\begin{proof}
Applying the involution $\iota$ of Lemma \ref{lem:IM} to Proposition \ref{prop:eKHe_sgn} gives
$$e_{\mathrm{sgn}}\mathcal{H}e_K = \iota\big(e_K \mathcal{H} e_{\mathrm{sgn}}\big) = \iota\big(\mathcal{A}^W \Omega^+ e_{\mathrm{sgn}}\big) = \mathcal{A}^W\, \Omega^- e_K,$$
since $\iota(\mathcal{A}^W) = \mathcal{A}^W$. Freeness holds because $\mathcal{A}$ is a domain and $F \mapsto Fe_K$ is injective on $\mathcal{A}$ (Lemma \ref{lem:HIK}). The last assertion holds because $\Omega^-e_K$ lies in $e_{\mathrm{sgn}}\mathcal{H}e_K$.
\end{proof}

Set
$$Z := \prod_{a \in \Sigma} f_a \in \mathcal{A}^W.$$
Under the standing assumption $\rho^\vee\in X$, this element factors as $Z=\Omega^{-}\Omega^{+}$.

\begin{theorem}\label{thm:genericity_Z}
Let $\chi$ be an algebra character of $\mathcal{H}_K$ and let $M = M_G(\chi)$ with spherical vector $v_K$. Then
$$e_{\mathrm{sgn}} M = \mathbb{C}\,\Omega^{-} v_K, \qquad \Omega^{-}v_K \neq 0 \iff \chi(Z) \neq 0.$$
\end{theorem}
\begin{proof}
For each $P\in\mathcal A^W$, we have $Pe_K\in\mathcal H_K$ and hence $Pv_K=(Pe_K)v_K=\chi(P)v_K$.

Put $w=\Omega^-v_K$. Then $e_{\mathrm{sgn}}w = w$ by the last assertion of Proposition \ref{prop:esgnHeK}. Moreover, the following calculation shows that $\mathcal A^Ww = \mathbb Cw$:
$$Pw=P\Omega^-v_K=\Omega^-Pv_K=\chi(P)w \qquad (P\in\mathcal A^W).$$
Since $M=\mathcal Hv_K$, Proposition \ref{prop:esgnHeK} now identifies the sign-isotypic part:
$$e_{\mathrm{sgn}}M=e_{\mathrm{sgn}}\mathcal He_Kv_K=\mathcal A^Ww=\mathbb Cw.$$

If $\chi(Z)\neq0$, then
$$\Omega^+w=Zv_K=\chi(Z)v_K\neq0,$$
so $w\neq0$.

Conversely, suppose $w\neq0$. Then $M=\mathcal Hw$ by simplicity. Since $w\in e_{\mathrm{sgn}}M$, Proposition \ref{prop:eKHe_sgn} implies
$$\mathbb Cv_K=e_KM=e_K\mathcal He_{\mathrm{sgn}}w=\mathcal A^W\Omega^+w=\mathbb C\,\chi(Z)v_K.$$
The left side is non-zero, so $\chi(Z)\neq0$.
\end{proof}

For $a\in\Sigma$, set
$$f_a^{\mathrm{inv}}:=\begin{cases}
1-q_a^{-1}\theta_{-a^\vee} & \text{if }a\text{ is not divisible},\\[2pt]
\big(1-c_1(a)^{-1}\theta_{-a^\vee}\big)\big(1+c_2(a)^{-1}\theta_{-a^\vee}\big) & \text{if }a\text{ is divisible}
\end{cases}.$$
In particular, for split $G$, $\prod_{a\in\Sigma}f_a^{\mathrm{inv}}=\prod_{\alpha\in\Phi}\big(1-q^{-1}\theta_{\alpha^\vee}\big)$.

\begin{corollary}[The genericity criterion, $\rho^\vee \in X$]\label{cor:genericity}
Assume $\rho^\vee \in X$ and let $\chi$ be an algebra character of $\mathcal{H}_K$. The following are equivalent.
\begin{enumerate}
\item $V_G(\chi)$ is $\psi$-generic for every character $\psi$ of $\overline{U}$ of conductor $\mathfrak{p}$.
\item $V_G(\chi)$ is $\omega$-generic for every character $\omega$ of $\overline{U}$ of conductor $\mathcal{O}$.
\item $V_G(\chi)$ is $\tau$-generic for some character $\tau$ of $\overline U$ of conductor $\mathcal O$ or $\mathfrak p$.
\item $\chi(Z) \neq 0$, equivalently $\displaystyle \chi\Big(\prod_{a\in\Sigma}f_a^{\mathrm{inv}}\Big) \neq 0$.
\end{enumerate}
\end{corollary}
\begin{proof}
Let $\psi$ be any character of $\overline{U}$ of conductor $\mathfrak{p}$. The genericity criterion in \cite[Corollary 5.2]{LuoGCS} states that
$$V_G(\chi)\text{ is }\psi\text{-generic}\iff e_{\mathrm{sgn}}M_G(\chi)\neq0.$$
By Theorem \ref{thm:genericity_Z}, the right side is equivalent to $\chi(Z)\neq0$. This proves the equivalence of (1), the alternative in (3) involving a character of conductor $\mathfrak{p}$, and (4).

Now let $\omega$ be a character of $\overline U$ of conductor $\mathcal{O}$ and set
$$\psi'(u):=\omega(t_{-\rho^\vee}ut_{\rho^\vee}).$$
The calculation in Lemma \ref{lem:conductor_shift} shows that $\psi'$ has conductor $\mathfrak p$. The map
$$\mathrm{Hom}_{\overline U}\big(V_G(\chi),\psi'\big)\longrightarrow
\mathrm{Hom}_{\overline U}\big(V_G(\chi),\omega\big),\qquad
\ell\longmapsto\big[v\mapsto\ell(t_{\rho^\vee}v)\big],$$
is an isomorphism. Thus the case of conductor $\mathcal O$ reduces to the case of conductor $\mathfrak p$, proving the remaining equivalences.

Finally, for each $a\in\Sigma^+$, pair the factors attached to $a$ and $-a$. Apply $(1-cx)(1-cx^{-1}) = c^2(1-c^{-1}x)(1-c^{-1}x^{-1})$ with $c=q_a$ when $a$ is not divisible, and with $c=c_1(a)$ and $c=-c_2(a)$ when $a$ is divisible. Since $c_1(a)^2c_2(a)^2 = q_a^2$, the accumulated constant is $\prod_{a\in\Sigma^+}q_a^2 = q_{w_0}^2$, and we obtain $Z = q_{w_0}^2\prod_{a\in\Sigma}f_a^{\mathrm{inv}}$. This proves the second form in (4).
\end{proof}

\begin{remark}\label{rem:Li}
Li's criterion \cite[Theorem 2.2]{Li92}, applied componentwise, is precisely condition (4). For a non-divisible relative root $\beta$, put $a=\tilde\beta\in\Sigma$. Li's parameters satisfy $q_{\beta/2}^{1/2}q_\beta=c_1(a)$ and $q_{\beta/2}^{1/2}=c_2(a)$, where $q_{\beta/2}=1$ on a reduced ray. Let $z_\chi:X\to\mathbb C^\times$ represent $\chi$, and write $z$ for the unramified quasicharacter occurring in Li's theorem. Our convention $t_\lambda=\lambda(\varpi^{-1})$ implies $z(a_\beta)=z_\chi(-a^\vee)$. Thus the zeta factor $\xi_\beta(z)$ in Li is $f_a^{\mathrm{inv}}$ evaluated at $z_\chi$.

The equivalence of the conductor $\mathfrak{p}$ and conductor $\mathcal{O}$ statements uses translation by $t_{\rho^\vee}$, hence the assumption $\rho^\vee\in X$. For general unramified $G$, Section \ref{sec:extended} proves the conductor $\mathcal{O}$ statement, while Remark \ref{rem:SL2} shows that the conductor $\mathfrak{p}$ statement may fail.
\end{remark}

\section{The Extended Group \texorpdfstring{$G''$}{G''} and the General Case}\label{sec:extended}

In this final section, $G$ is an arbitrary unramified group over $F$ as in Section \ref{sec:notation}, and the assumption $\rho^\vee \in X$ is dropped. The notation from Section \ref{sec:notation} remains in force. Lemma \ref{lem:whittaker_support} was stated in this generality, and Lemma \ref{lem:pairing} does not use $\rho^\vee \in X$. We use both freely below, together with the lattice $X''$ and the algebra $\mathcal{A}''$ of Section \ref{subsec:chV}.

What fails without $\rho^\vee \in X$ is the existence of the translation element $t_{\rho^\vee}$, of the weights $\theta_{\pm\rho^\vee}$, and of the factors $\Omega^{\pm}$. We restore them by embedding the situation into a central quotient group $G''$ whose cocharacter lattice contains $\rho^\vee$.

\subsection{The extended group}

Recall that $E/F$ is the minimal splitting extension of $G$, with Frobenius generator $\sigma$.
Set
$$X_T'' := X_*(T)+\mathbb Z\rho^\vee
\subseteq X_*(T)\otimes_{\mathbb Z}\mathbb Q,\qquad
X'' := X+\mathbb Z\rho^\vee
\subseteq X\otimes_{\mathbb Z}\mathbb Q.$$
The lattice $X_T''$ is $\sigma$-stable and $(X_T'')^\sigma=X''$ because $\sigma$ fixes $\rho^\vee$ and $X_*(T)^\sigma=X$. Moreover, $2\rho^\vee$ is a sum of absolute coroots, so both $X_T''/X_*(T)$ and $X''/X$ have order at most two.

Let $\Lambda:=X^*(T):=\operatorname{Hom}_{E\text{-grp}}(T_E,\mathbb G_{m,E})$, and let $\Lambda'':=\{y\in\Lambda\mid \langle y,\rho^\vee\rangle\in\mathbb Z\}$ be the lattice dual to $X_T''$. It is $\sigma$-stable and satisfies $[\Lambda:\Lambda'']=[X_T'':X_*(T)]$.
Every absolute root lies in $\Lambda''$. Indeed, an absolute root is an integral combination of absolute simple roots, each of which pairs with $\rho^\vee$ to $1$. The absolute coroots already lie in $X_*(T)\subseteq X_T''$.

By Galois descent and the anti-equivalence between groups of multiplicative type and their character groups, the finite abelian group $\Lambda/\Lambda''$ with its induced $\sigma$-action determines a finite flat $\mathcal O$-group $C_{\mathcal O}$ of multiplicative type \cite[Proposition B.3.4]{ConradRGS}. The quotient $\Lambda\twoheadrightarrow\Lambda/\Lambda''$ induces a closed immersion $C_{\mathcal O}\hookrightarrow T_{\mathcal O}$ by \cite[Lemma B.1.3]{ConradRGS}. Every absolute root restricts trivially to $(C_{\mathcal O})_{\mathcal O_E}$ because it belongs to $\Lambda''$. Thus $(C_{\mathcal O})_{\mathcal O_E}$ lies in the scheme-theoretic center of $G_{\mathcal O_E}$ by \cite[Example 5.1.7]{ConradRGS}, and descent shows that $C_{\mathcal O}$ is central in $G_{\mathcal O}$.

Let $C\subseteq T$ be the generic fiber of $C_{\mathcal O}$. It has character group $\Lambda/\Lambda''$ and order $[\Lambda:\Lambda'']$. Scheme-theoretically, $C_E=\bigcap_{y\in\Lambda''}\ker\bigl(y:T_E\to\mathbb G_m\bigr)$.

Define $G_{\mathcal O}'':=G_{\mathcal O}/C_{\mathcal O}$ and $T_{\mathcal O}'':=T_{\mathcal O}/C_{\mathcal O}$, and let $p_{\mathcal O}:G_{\mathcal O}\to G_{\mathcal O}''$ be the quotient map.

Let $G''$, $T''$, and $p:G\to G''$ be the respective generic fibers. Put $S'':=p(S)$, $U'':=p(U)$, $\overline U'':=p(\overline U)$, $B'':=B/C$, and $\overline B'':=\overline B/C$.

Let $x_\gamma:\mathbb G_{a,\mathcal O_E}\xrightarrow{\sim}U_{\gamma,\mathcal O_E}$, $\gamma\in\Phi_{\mathrm{abs}}$, be the absolute root-group coordinates of the $\sigma$-equivariant Chevalley--Steinberg system fixed in Section \ref{sec:notation}. For each absolute root $\gamma$, let $x_\gamma'':=p_{\mathcal O_E}\circ x_\gamma$, where $p_{\mathcal O_E}$ is the base change of $p_{\mathcal O}$ to $\mathcal O_E$.

\begin{lemma}\label{lem:extended_group}
The group scheme $G_{\mathcal O}''$ is reductive with maximal torus $T_{\mathcal O}''$, and the maps $x_\gamma''$ form compatible $\sigma$-equivariant root-group coordinates over $\mathcal O_E$. Its generic fiber $G''$ is an unramified connected reductive group over $F$ with maximal torus $T''$, maximal split torus $S''$, and Borel subgroup $B''$. Moreover,
$$X_*(T'') = X_T'', \qquad X^*(T'') = \Lambda'', \qquad X_*(S'') = X'',$$
and $G$ and $G''$ have the same absolute and relative root systems, coroots, Weyl group $W$, system $\Sigma$, and parameters. The central isogeny $p$ restricts to an isomorphism
$U_\beta\xrightarrow{\sim}U_\beta''$ for every non-divisible relative root $\beta$, compatibly with the root-group filtrations. It therefore induces isomorphisms $U\xrightarrow{\sim}U''$ and $\overline U\xrightarrow{\sim}\overline U''$.
\end{lemma}
\begin{proof}
By \cite[Corollary 3.3.5]{ConradRGS}, $G_{\mathcal O}''$ is reductive with maximal torus $T_{\mathcal O}''$. Character duality gives $X^*(T'')=\Lambda''$ and $X_*(T'')=\operatorname{Hom}(\Lambda'',\mathbb Z)=X_T''$. The map on character lattices is the inclusion $\Lambda''\hookrightarrow\Lambda$, so the absolute roots and coroots are unchanged. The image $S''$ of $S$ is a maximal split torus with $X_*(S'')=(X_T'')^\sigma=X''$.
Moreover, $B''=B/C$ is an $F$-Borel subgroup and $G_E''$ is split. Thus $G''$ is unramified.

Over $\mathcal O_E$, the central isogeny $p_{\mathcal O_E}$ restricts to isomorphisms between corresponding root subgroup schemes \cite[Lemma 6.1.4 and Example 6.1.9]{ConradRGS}. The maps $x_\gamma''$ are therefore compatible $\sigma$-equivariant root-group coordinates. After descent, they identify $U_\beta$ with $U_\beta''$ for every non-divisible relative root $\beta$. They also identify the filtrations on $U_\beta/U_{2\beta}$ and on $U_{2\beta}$ when the latter occurs. Hence the relative root systems, the fields $E_\beta$, the affine-root filtrations, and all Iwahori--Hecke parameters, including $q_{a*}$, agree. The root-group isomorphisms assemble to the asserted isomorphisms on $U$ and $\overline U$.
\end{proof}

We mark objects attached to $G''$ with double primes. Put $K'':=G_{\mathcal O}''(\mathcal O)$ and $t_\lambda'':=\lambda(\varpi^{-1})$ for $\lambda\in X''$, and let $I''$ be the Iwahori subgroup satisfying $I''\cap B''=K''\cap B''$. Let $\mathcal H''$ and $\mathcal H_{K''}''$ be the corresponding Iwahori and spherical Hecke algebras, and write $e_{K''}\in\mathcal H''$ for the spherical idempotent.

For $\lambda\in(X'')^-$, define $\theta''_\lambda:=(q''_{t''_\lambda})^{-1/2}T''_{t''_\lambda}$ and extend the definition multiplicatively to $X''$. The map $\theta_\lambda\mapsto\theta_\lambda''$ identifies the abstract group algebra $\mathcal A''$ of Section \ref{subsec:chV} with the Bernstein subalgebra $\operatorname{span}_{\mathbb C}\{\theta_\lambda'':\lambda\in X''\}\subseteq\mathcal H''$.

The quotient map satisfies $p(K)\subseteq K''$, $p(I)\subseteq I''$, and $p(t_\lambda)=t_\lambda''$ for $\lambda\in X$.
Normalize the two Haar measures by $\operatorname{vol}(I)=\operatorname{vol}(I'')=1$. Lemma \ref{lem:extended_group} then shows that
$$W''(q)=W(q),\qquad
\operatorname{vol}(K)=\operatorname{vol}(K'')=W(q),\qquad
\delta_{\overline B''}(t_y'')=\delta_{\overline B}(t_y)\quad(y\in X).$$

Since $\rho^\vee \in X''$, all results of Sections \ref{sec:bbf}--\ref{sec:cs} apply to $G''$.

\subsection{The Hecke algebra embedding}

Together with Lemma \ref{lem:extended_group}, the descriptions of the Iwahori--Weyl group and its affine Coxeter structure in \cite[Section 2.9 and Proposition 3.2.2]{Rostami} identify
$$\widetilde{W} = W \ltimes X \ \subseteq\ W \ltimes X'' = \widetilde{W}''.$$
The two groups have the same affine Coxeter subgroup and parameters, and the length function on $\widetilde W''$ restricts to that on $\widetilde W$. The Iwahori--Matsumoto bases define an injective linear map
$$j:\mathcal H\longrightarrow\mathcal H'',\qquad j(T_w):=T_w''\quad(w\in\widetilde W).$$
The Iwahori--Matsumoto relations in \cite[Proposition 4.1.1]{Rostami} are preserved by $j$, so $j$ is an algebra homomorphism.

\begin{lemma}\label{lem:j_embedding}
The restriction of the algebra embedding $j$ to $\mathcal A$ is the natural inclusion $j(\theta_\lambda)=\theta''_\lambda$ for $\lambda\in X$. Moreover, $j$ has the following properties.
\begin{enumerate}
\item $j(e_K) = e_{K''}$ and $j(F e_K) = F e_{K''}$ for $F \in \mathcal{A}^W$. Thus $j$ restricts to a unital algebra embedding $j|_{\mathcal{H}_K} : \mathcal{H}_K \hookrightarrow \mathcal{H}''_{K''}$. Under the Satake isomorphisms of Theorem \ref{thm:spherical_iso}, this embedding corresponds to the inclusion $\mathcal{A}^W \subseteq (\mathcal{A}'')^W$.
\item $j(\mathbf{1}_{K t_\mu K}) = \mathbf{1}_{K'' t''_\mu K''}$ for every $\mu \in X^+$.
\end{enumerate}
\end{lemma}
\begin{proof}
For $\lambda \in X^-$, the definitions and the equality $q_{t_\lambda} = q''_{t''_\lambda}$ imply $j(\theta_\lambda) = \theta''_\lambda$. Multiplicativity extends this to all $\lambda \in X$.

(1) The definition of $e_K$ gives $j(e_K) = e_{K''}$, and the preceding description of $j|_{\mathcal A}$ gives $j(Fe_K)=Fe_{K''}$. Theorem \ref{thm:spherical_iso} then identifies the restriction of $j$ with the inclusion $\mathcal A^W\subseteq(\mathcal A'')^W$.

(2) The two Cartan double cosets decompose into Iwahori double cosets indexed by the same set $Wt_\mu W \subseteq \widetilde{W}$:
$$K t_\mu K = \bigsqcup_{w \in W t_\mu W} IwI, \qquad K'' t''_\mu K'' = \bigsqcup_{w \in W t_\mu W} I''wI''.$$
Applying $j(\mathbf 1_{IwI})=\mathbf 1_{I''wI''}$ termwise proves (2).
\end{proof}

\begin{lemma}\label{lem:cartan_preimage}
For every $\mu \in X^+$, we have $p^{-1}\big(K'' t''_\mu K''\big) = K t_\mu K$ inside $G(F)$. In particular, $p^{-1}(K'') = K$, and for $h \in \mathcal{H}_K$,
$$h(g) = j(h)\big(p(g)\big) \qquad \text{for all } g \in G(F).$$
Moreover, $g K \mapsto p(g)K''$ is a bijection from $K t_\mu K / K$ onto $K'' t''_\mu K'' / K''$.
\end{lemma}
\begin{proof}
The inclusion $p(Kt_\mu K) \subseteq K''t''_\mu K''$ is clear. Conversely, if $p(g) \in K''t''_\mu K''$ and $g \in Kt_\lambda K$ with $\lambda \in X^+$, then $p(g) \in K''t''_\lambda K''$, and disjointness of the Cartan decomposition of $G''(F)$ over $(X'')^+ \supseteq X^+$ forces $\lambda = \mu$. The identity $h = j(h)\circ p$ follows for $h = \mathbf{1}_{Kt_\mu K}$ from Lemma \ref{lem:j_embedding}(2) and $p^{-1}(K''t''_\mu K'') = Kt_\mu K$, and extends linearly.

For the last claim, injectivity holds because $p^{-1}(K'') = K$. The Iwahori double-coset decompositions and $\mathrm{vol}(IwI) = q_w$, with the same index set $Wt_\mu W$ and the same parameters on both sides, lead to
$$\mathrm{vol}_G(Kt_\mu K) = \sum_{w \in Wt_\mu W} q_w = \mathrm{vol}_{G''}(K''t''_\mu K'').$$
Dividing by $\mathrm{vol}(K) = W(q) = \mathrm{vol}(K'')$ shows that both coset spaces have the same finite cardinality, so the injection is a bijection.
\end{proof}

\begin{corollary}\label{cor:hecke_action_compatibility}
Let $\pi''$ be a smooth representation of $G''(F)$ and let $v \in \pi''$ be fixed by $K''$. View $\pi'' \circ p$ as a smooth representation of $G(F)$. Then for every $h \in \mathcal{H}_K$,
$$(\pi''\circ p)(h)\, v = \pi''\big(j(h)\big)\, v.$$
\end{corollary}
\begin{proof}
By linearity, it suffices to take $h=\mathbf{1}_{Kt_\mu K}$. The asserted identity follows from Lemma \ref{lem:cartan_preimage} and the equality $\mathrm{vol}(K)=\mathrm{vol}(K'')$.
\end{proof}

\subsection{Transfer of freeness and the general Casselman--Shalika formula}

Let $\omega$ be an arbitrary character of $\overline U$ of conductor $\mathcal O$. Since $p|_{\overline U}:\overline U\to\overline U''$ is an isomorphism, define
$$\omega''(p(u)):=\omega(u)\qquad(u\in\overline U).$$
The filtration compatibility in Lemma \ref{lem:extended_group} shows that $\omega''$ also has conductor $\mathcal O$.

Let $\Phi^\omega_y$, $y\in X^+$, be the basis functions defined in Lemma \ref{lem:whittaker_support}(2) for $G$, and put $\Phi''_z:=\Phi^{\omega''}_z$ for $z\in(X'')^+$. The bases in that part define an injective linear map
$$\iota_V : (\mathrm{ind}_{\overline{U}}^G\,\omega)^K \hookrightarrow (\mathrm{ind}_{\overline{U}''}^{G''}\,\omega'')^{K''}, \qquad \iota_V(\Phi^\omega_y) := \Phi''_y \quad (y \in X^+),$$
a correspondence of basis elements, not composition with $p$, since $p(K)$ may be a proper subgroup of $K''$.

\begin{lemma}\label{lem:transfer}
For $h \in \mathcal H_K$ and $f \in (\mathrm{ind}_{\overline U}^G\omega)^K$,
$$\iota_V(f*h)=\iota_V(f)*j(h).$$
\end{lemma}
\begin{proof}
By linearity, it suffices to take $f=\Phi^\omega_x$ and $h=h_\mu$ with $x,\mu\in X^+$. Normalize the Haar measures on $T(F)$ and $T''(F)$ by
$$\operatorname{vol}\bigl(T(\mathcal O)\bigr)
=\operatorname{vol}\bigl(T''(\mathcal O)\bigr)=1,$$
and choose the measures on $\overline U(F)$ and $\overline U''(F)$ compatibly with the respective Iwasawa integration formulas. Applying these formulas to $\mathbf 1_K$ and $\mathbf 1_{K''}$ shows that
$$\operatorname{vol}\bigl(\overline U(\mathcal O)\bigr)
=\operatorname{vol}\bigl(\overline U''(\mathcal O)\bigr)=1.$$
Lemma \ref{lem:extended_group} identifies these compact open subgroups, so $p$ carries the first unipotent measure to the second. For $y\in X$, unfolding the convolution using the Iwasawa integration formula produces
$$\big(\Phi^\omega_x*h_\mu\big)(t_y)
=W(q)\,\delta_{\overline B}(t_x)^{-1/2}
\int_{\overline U}h_\mu(t_{-x}\bar u^{-1}t_y)\omega(\bar u)\,d\bar u.$$
The corresponding formula for $G''$ yields
$$\big(\Phi^\omega_x*h_\mu\big)(t_y)=\bigl(\Phi_x''*j(h_\mu)\bigr)(t_y'').$$
Indeed, Lemma \ref{lem:cartan_preimage} identifies the two integrands under $p$, while $\omega''\circ p=\omega$ by the definition of $\omega''$, and $\delta_{\overline B''}(t_x'')=\delta_{\overline B}(t_x)$ by Lemma \ref{lem:extended_group}.

It remains to show that $\Phi_x''*j(h_\mu)$ has no basis component indexed by $(X'')^+\setminus X^+$. The proof of Corollary \ref{cor:freeness_O} for $G''$, applied to $\omega''$, and Lemma \ref{lem:j_embedding}(1) produce
$$\Phi_x''*j(h_\mu)
=\Phi_0''*\bigl(\mathrm{ch}V_x\,\mathrm{ch}V_\mu\,e_{K''}\bigr)
=\sum_{\nu\in(X'')^+}d_\nu\Phi_\nu'',$$
where $\mathrm{ch}V_x\,\mathrm{ch}V_\mu=\sum_\nu d_\nu\,\mathrm{ch}V_\nu$. This product belongs to $\mathcal A^W$, so Fact \ref{fact:chV}(2) implies $d_\nu=0$ unless $\nu\in X^+$. The preceding comparison of values therefore identifies the two basis expansions and proves the lemma.
\end{proof}

\begin{theorem}[Freeness, general case]\label{thm:freeness_general}
Let $G$ be an arbitrary unramified group and let $\omega$ be any character of $\overline U$ of conductor $\mathcal O$. Then for every $y \in X^+$,
$$\Phi^\omega_0 * h_y = \Phi^\omega_y,$$
and $(\mathrm{ind}_{\overline{U}}^G\,\omega)^K$ is a free right $\mathcal{H}_K$-module of rank one generated by $\Phi^\omega_0$.
\end{theorem}
\begin{proof}
Since $\rho^\vee \in X''$, the proof of Corollary \ref{cor:freeness_O} for $G''$, applied to $\omega''$, establishes
$$\Phi''_0 * \big(\mathrm{ch}V_y\, e_{K''}\big) = \Phi''_y \qquad \big(y \in X^+ \subseteq (X'')^+\big).$$
Fact \ref{fact:chV}(2) shows that $\mathrm{ch}V_y \in \mathcal{A}^W \subseteq (\mathcal{A}'')^W$. By Lemma \ref{lem:j_embedding}(1), $j(h_y) = \mathrm{ch}V_y\, e_{K''}$, so Lemma \ref{lem:transfer} proves
$$\iota_V\big(\Phi^\omega_0 * h_y\big) = \Phi''_0 * \big(\mathrm{ch}V_y e_{K''}\big) = \Phi''_y = \iota_V\big(\Phi^\omega_y\big).$$
Since $\iota_V$ is injective, $\Phi^\omega_0 * h_y = \Phi^\omega_y$. Freeness follows because $\{h_y\mid y\in X^+\}$ is a basis of $\mathcal H_K$ by Fact \ref{fact:chV}(2) and Theorem \ref{thm:spherical_iso}, and $\{\Phi^\omega_y\mid y\in X^+\}$ is a basis of $(\mathrm{ind}_{\overline U}^G\omega)^K$ by Lemma \ref{lem:whittaker_support}(2).
\end{proof}

\begin{theorem}[Casselman--Shalika formula, general case]\label{thm:CS_general}
Let $G$ be an arbitrary unramified group, let $\omega$ be any character of $\overline U$ of conductor $\mathcal O$, and let $0 \neq \mathcal{W} \in (\mathrm{Ind}_{\overline{U}}^G\, \omega)^K$ be a spherical Whittaker function with parameter $\chi$. Then $\mathcal{W}(t_y) = 0$ for $y \notin X^+$, and
\[
\mathcal{W}(t_y) = \delta_{\overline{B}}(t_y)^{1/2}\, \chi(\mathrm{ch}V_y)\, \mathcal{W}(e) \qquad (y \in X^+).
\]
In particular, $\mathcal{W}(e) \neq 0$.
\end{theorem}
\begin{proof}
The proof of Theorem \ref{thm:CS_O} applies verbatim after replacing Corollary \ref{cor:freeness_O} by Theorem \ref{thm:freeness_general} for $\overline\omega$.
\end{proof}

\subsection{The general genericity criterion}

Recall the central element $Z = \prod_{a\in\Sigma}f_a$ defined in Section \ref{sec:cs}. It lies in $\mathcal{A}^W$ for every unramified $G$, since the coroots lie in $X$ and the parameters are $W$-invariant. Only its factorization $Z = \Omega^-\Omega^+$ required $\rho^\vee \in X$.

\begin{lemma}\label{lem:functional_nonvanishing}
Let $\omega$ be a character of $\overline U$ of conductor $\mathcal O$. Every non-zero $\omega$-Whittaker functional $\lambda$ on $V_G(\chi)$ satisfies $\lambda(v_K)\neq0$.
\end{lemma}
\begin{proof}
Define $\mathcal W(g):=\lambda(gv_K)$. This is a spherical Whittaker function with parameter $\chi$. It is non-zero because $v_K$ generates the irreducible representation $V_G(\chi)$, so Theorem \ref{thm:CS_general} implies $\lambda(v_K)=\mathcal W(e)\neq0$.
\end{proof}

\begin{theorem}[Genericity, general case]\label{thm:genericity_general}
Let $G$ be an arbitrary unramified group over $F$, let $\chi$ be an algebra character of $\mathcal{H}_K$, and let $\omega$ be any character of $\overline U$ of conductor $\mathcal O$. Then $V_G(\chi)$ is $\omega$-generic if and only if
$$\chi(Z) \neq 0, \qquad \text{equivalently} \qquad \chi\Big(\prod_{a\in\Sigma}f_a^{\mathrm{inv}}\Big) \neq 0,$$
with $f_a^{\mathrm{inv}}$ as in Corollary \ref{cor:genericity}.
\end{theorem}
\begin{proof}
Let $\omega''$ be the character of $\overline U''$ defined by $\omega''(p(u))=\omega(u)$ for $u\in\overline U$. It has conductor $\mathcal O$ by Lemma \ref{lem:extended_group}. Choose a homomorphism $z:X\to\mathbb C^\times$ representing $\chi$. Since $\mathbb C^\times$ is divisible, $z$ extends to a homomorphism $z'':X''\to\mathbb C^\times$. Let $\chi''$ be the corresponding algebra character of $\mathcal H_{K''}''\cong(\mathcal A'')^W$. Then $\chi''\circ j|_{\mathcal H_K}=\chi$ and $\chi''(Z)=\chi(Z)$.

Put $\pi''=V_{G''}(\chi'')$. The isogeny-restriction theorem of Silberger \cite{Sil79} gives a finite direct-sum decomposition
$$\pi''\circ p=\bigoplus_{i=1}^m\pi_i$$
into irreducible admissible representations of $G(F)$. Write the spherical vector as $v_{K''}=\sum_i v_i$ with $v_i\in\pi_i$. Every $v_i$ is fixed by $K$. For $h\in\mathcal H_K$, Corollary \ref{cor:hecke_action_compatibility} shows that
$$
(\pi''\circ p)(h)v_{K''}
=\pi''\bigl(j(h)\bigr)v_{K''}
=\chi''\bigl(j(h)\bigr)v_{K''}
=\chi(h)v_{K''}.
$$
Projecting this equality to $\pi_i$, we obtain
$$\pi_i(h)v_i=\chi(h)v_i\qquad(h\in\mathcal H_K).$$
Thus Fact \ref{fact:borel} and Lemma \ref{lem:spherical_unique} show that
$$\pi_i\cong V_G(\chi)\qquad\text{whenever }v_i\neq0.$$

Suppose first that $\chi(Z)\neq0$. Corollary \ref{cor:genericity} for $G''$ gives a non-zero $\omega''$-Whittaker functional $\lambda''$ on $\pi''$. By Lemma \ref{lem:functional_nonvanishing}, $\lambda''(v_{K''})\neq0$, so $\lambda''(v_i)\neq0$ for some $i$. The restriction of $\lambda''$ to $\pi_i$ is non-zero and, for $u\in\overline U$ and $v\in\pi_i$, satisfies $\lambda''(uv)=\omega''(p(u))\lambda''(v)=\omega(u)\lambda''(v)$. It is therefore an $\omega$-Whittaker functional. Since $v_i\neq0$, Fact \ref{fact:borel} and Lemma \ref{lem:spherical_unique} identify $\pi_i$ with $V_G(\chi)$, and hence $V_G(\chi)$ is $\omega$-generic.

Conversely, suppose that $V_G(\chi)$ is $\omega$-generic. Choose $i$ with $v_i\neq0$ and identify $\pi_i$ with $V_G(\chi)$. If $\lambda_i$ is a non-zero $\omega$-Whittaker functional on $\pi_i$ and $\operatorname{pr}_i:\pi''\to\pi_i$ is the $G(F)$-equivariant projection, then $\lambda_i\circ\operatorname{pr}_i$ is a non-zero $\omega''$-Whittaker functional on $\pi''$, because $p:\overline U(F)\xrightarrow{\sim}\overline U''(F)$. Corollary \ref{cor:genericity} for $G''$ now implies $\chi(Z)=\chi''(Z)\neq0$.

This proves the genericity criterion. The second form follows from $Z=q_{w_0}^2\prod_{a\in\Sigma}f_a^{\mathrm{inv}}$ as in the proof of Corollary \ref{cor:genericity}.
\end{proof}

\begin{remark}[Conductor $\mathfrak{p}$ fails in general]\label{rem:SL2}
Take $G=\mathrm{SL}_2$, so $X=\mathbb Z\alpha^\vee$, $\rho^\vee=\tfrac12\alpha^\vee\notin X$, and $G''=\mathrm{PGL}_2$. Put $b:=\theta''_{\rho^\vee}$. Then
$$\mathcal A=\mathbb C[b^{\pm2}],\qquad \mathcal A''=\mathbb C[b^{\pm1}],\qquad
\Omega^+=b(1-qb^{-2}),\qquad \Omega^-=b^{-1}(1-qb^2).$$
The ring $\mathcal A$ is the fixed ring of $b\mapsto-b$, whereas $\Omega^-$ is anti-invariant. Since the odd part of $(\mathcal A'')^W=\mathbb C[b+b^{-1}]$ is $(b+b^{-1})\mathcal A^W$, it follows that
$$(\mathcal A'')^W\Omega^-\cap\mathcal A=\mathcal A^W(b+b^{-1})\Omega^-.$$
Lemma \ref{lem:j_embedding} and Proposition \ref{prop:esgnHeK}, applied to $G''$, now yield the first identity below. The symmetric argument using Proposition \ref{prop:eKHe_sgn} yields the second:
$$e_{\mathrm{sgn}}\mathcal He_K=\mathcal A^W\Omega_1^-e_K,\qquad
e_K\mathcal He_{\mathrm{sgn}}=\mathcal A^W\Omega_1^+e_{\mathrm{sgn}},$$
where
\begin{align*}
\Omega_1^-&:=(b+b^{-1})\Omega^-=(1+\theta_{-\alpha^\vee})(1-q\theta_{\alpha^\vee}),\\
\Omega_1^+&:=(b+b^{-1})\Omega^+=(1+\theta_{\alpha^\vee})(1-q\theta_{-\alpha^\vee}).
\end{align*}
The proof of Theorem \ref{thm:genericity_Z} uses only these two identities and therefore shows
$$e_{\mathrm{sgn}}M_G(\chi)\ne0\quad\Longleftrightarrow\quad\chi(Z_1)\ne0,\qquad
Z_1:=\Omega_1^-\Omega_1^+=(2+\theta_{\alpha^\vee}+\theta_{-\alpha^\vee})Z.$$
Let $\chi$ be represented by $z:X\to\mathbb C^\times$ with $z(\alpha^\vee)=-1$. Then $\chi(Z)=(1+q)^2\ne0$, whereas $\chi(Z_1)=0$. Thus Theorem \ref{thm:genericity_general} makes $V_G(\chi)$ $\omega$-generic for every character $\omega$ of $\overline U$ of conductor $\mathcal O$, while \cite[Corollary 5.2]{LuoGCS} shows that it admits no Whittaker functional of conductor $\mathfrak p$. Hence the equivalence between the two conductors in Corollary \ref{cor:genericity} can fail when $\rho^\vee\notin X$.
\end{remark}

\par
\nobreak
\begingroup
\footnotesize
\interlinepenalty=10000
\par\addvspace{\bigskipamount}\indent
{\scshape Department of Mathematics, University of Utah, Salt Lake City, UT 84112}\par
\nobreak\indent{\itshape Email address}\/:\space
\texttt{yiluo@math.utah.edu}\par
\endgroup

\end{document}